\documentclass[onefignum,onetabnum]{siamart171218}

\usepackage{lipsum}
\usepackage{amsfonts}
\usepackage{graphicx}
\usepackage{epstopdf}
\usepackage{algorithm}
\usepackage{algorithmic}
\usepackage{tikz}

\usepackage{multirow}
\usepackage{listings}
\usepackage{mathtools}
\usepackage{latexsym,amsmath,amsfonts,amscd}
\usepackage{subfigure}
\usepackage{bbm}

\newtheorem{thm}{Theorem}[section]
\newtheorem{lem}[thm]{Lemma}

\newtheorem{rmk}[thm]{Remark}

\newtheorem{ex}{Example}

\ifpdf
  \DeclareGraphicsExtensions{.eps,.pdf,.png,.jpg}
\else
  \DeclareGraphicsExtensions{.eps}
\fi

\headers{A Bound-Preserving Compact Finite Difference Scheme}{H. Li, S. Xie and X. Zhang}

\title{A high order accurate bound-preserving compact finite difference scheme for scalar convection diffusion equations
\thanks{H. Li and X. Zhang  were supported by the NSF grant DMS-1522593. 
S. Xie was supported by NSFC grant 11371333 and Fundamental Research Funds for the Central Universities 201562012. }}

\author{Hao Li\thanks{Department of Mathematics,
Purdue University,
150 N. University Street,
West Lafayette, IN 47907-2067
  (\email{li2497@purdue.edu}, \email{zhan1966@purdue.edu}).}
\and Shusen Xie
\thanks{School of Mathematical Sciences,
Ocean University of China,
238 Songling Road,
Qingdao 266100, PR China (\email{shusenxie@ouc.edu.cn})}
  \and Xiangxiong Zhang \footnotemark[2]}

\usepackage{amsopn}

\begin{document}

\maketitle

\begin{abstract}
We show that the classical fourth order accurate compact finite difference scheme with high order strong stability preserving time discretizations for 
convection diffusion problems satisfies a weak monotonicity property, which implies that a simple  limiter can enforce 
the bound-preserving property without losing conservation and high order accuracy. Higher order accurate compact finite difference schemes satisfying the weak monotonicity will also be discussed. 
\end{abstract}



\section{Introduction}
\label{sec1}
\setcounter{equation}{0}
\setcounter{figure}{0}
\setcounter{table}{0}
\subsection{The bound-preserving property}

Consider the initial value problem for a  scalar convection diffusion equation
$u_t+f(u)_x=a(u)_{xx},\quad u(x,0)=u_0(x),$
where $a'(u)\geq 0$.  Assume $f(u)$ and $a(u)$ 
are well-defined smooth functions for any $u\in[m, M]$ where  $m=\min_x u_0(x)$ and $M=\max_x u_0(x)$. 
Its exact solution satisfies:
\begin{equation}
\min_x u_0(x)=m\leq u(x,t)\leq M=\max_x u_0(x),\quad \forall t\geq 0.
\label{bp-property}
\end{equation}
In this paper, we are interested in constructing a high order accurate finite difference scheme 
satisfying the bound-preserving property  \eqref{bp-property}. 

For a scalar problem, it is desired to achieve \eqref{bp-property}
in numerical solutions mainly for the physical meaning. For instance, if $u$ denotes density and $m=0$, then negative numerical solutions are meaningless. 
In practice, in addition to enforcing \eqref{bp-property}, it is also critical to strictly enforce  the global conservation of numerical solutions for a time-dependent convection 
dominated problem.  
Moreover, the computational cost for enforcing \eqref{bp-property} should not be significant if it  is needed for each time
step.

\subsection{Popular methods for convection problems}
For the convection problems, i.e., $a(u)\equiv 0$, 
a straightforward way to achieve the above goals is to require a scheme to be monotone, total-variational-diminishing (TVD), or satisfying a discrete maximum principle, 
which all imply the bound-preserving property.
But most schemes satisfying these stronger properties are at most second order accurate. 
For instance, a monotone scheme and traditional TVD finite difference and finite volume schemes are at most first order accurate \cite{leveque1992}.
Even though it is possible to have high order TVD finite volume schemes in the sense of measuring the total variation of reconstruction polynomials 
\cite{sanders1988third, zhang2010genuinely}, such schemes can be constructed only for the one-dimensional problems. 
The second order central scheme satisfies a discrete  maximum principle $\min_j u_j^n\leq u_j^{n+1}\leq \max_j u_j^n$ where 
$u^n_j$ denotes the numerical solution at $n$-th time step and $j$-th grid point \cite{levy1997non}.
Any finite difference scheme satisfying such a maximum principle can be at most second order accurate, see Harten's example in \cite{zhang2011maximum}.
By measuring the extrema of reconstruction polynomials, third order maximum-principle-satisfying schemes can be constructed  \cite{liu1996nonoscillatory}
but extensions to multi-dimensional nonlinear
problems are very difficult.

For constructing high order accurate schemes, one can enforce only the bound-preserving property for fixed known bounds, e.g., $m=0$ and $M=1$ if $u$
denotes the density ratio. Even though high order linear schemes cannot be monotone,  high order finite volume type spatial discretizations including the discontinuous 
Galerkin (DG) method satisfy a weak monotonicity property \cite{zhang2010maximum,zhang2011maximum, zhang2012maximum}. Namely, in a scheme consisting of any high order 
finite volume spatial discretization and forward Euler time discretization,
the cell average is a monotone function of the point values of the reconstruction or approximation polynomial at Gauss-Lobatto quadrature points. 
Thus if these point values are in the desired range $[m,M]$, so are the cell averages in the next time step. 
A simple and efficient local bound-preserving limiter can be designed to control these point values without destroying conservation. 
Moreover, this simple limiter is high order accurate, see \cite{zhang2010maximum} and the appendix in \cite{zhang2017positivity}. 
With strong stability preserving (SSP) Runge-Kutta or multistep methods \cite{gottlieb2011strong}, which are 
convex combinations of several formal forward Euler steps, a high order accurate finite volume or DG scheme can be rendered bound-preserving with this limiter. 
These results can be easily extended to multiple dimensions on cells of general shapes.
However, for a general finite difference scheme, the weak monotonicity does not hold. 

For enforcing only the bound-preserving property in high order schemes,  efficient alternatives include a  flux limiter \cite{xu2014parametrized, xiong2015high}
and a sweeping limiter in 
\cite{liu2017simple}. These methods are designed to directly enforce the bounds  without destroying conservation thus can be used on any conservative schemes. Even though they
work well in practice,
it is nontrivial to analyze and rigorously justify the accuracy of these methods especially for multi-dimensional nonlinear problems. 

\subsection{The weak monotonicity in compact finite difference schemes}
\label{sec-intro-monotonicity}
Even though the weak monotonicity does not hold for a general  finite difference scheme, in this paper we will show that some 
high order compact finite difference schemes  
satisfy such a property, which implies a simple limiting procedure can be used to 
enforce bounds without destroying accuracy and conservation. 

To demonstrate the main idea, we first consider a fourth order accurate compact finite difference approximation to the first derivative on the interval $[0,1]$:
\begin{equation*}
\frac16( f'_{i+1}+4f'_i +f'_{i-1})=\frac{f_{i+1}-f_{i-1}}{2\Delta x}+\mathcal{O} (\Delta x^4),
\end{equation*}
where $f_i$ and $f'_i$ are point values of a function $f(x)$ and its derivative $f'(x)$ at uniform grid points $x_i$ $(i=1,\cdots,N)$ respectively.
For periodic boundary conditions, the following tridiagonal linear system needs to be solved to obtain the implicitly defined approximation to the first order derivative:
{\setlength\arraycolsep{2pt}
\begin{equation}
\frac16\begin{pmatrix}
              4 & 1 & & &  1\\
              1 & 4 & 1 & & \\  
               & \ddots & \ddots & \ddots & \\
               & & 1 & 4 &1 \\
              1  & & & 1 & 4 \\
             \end{pmatrix} 
 \begin{pmatrix}
 f'_1 \\
 f'_2\\
 \vdots\\
 f'_{N-1}\\
 f'_N
 \end{pmatrix} =
 \frac{1}{2\Delta x}\begin{pmatrix}
              0 & 1 & & &  -1\\
              -1 & 0 & 1 & &  \\ 
               & \ddots & \ddots & \ddots & \\
               & & -1 & 0 & 1 \\
              1  & & & -1 & 0 \\
             \end{pmatrix} 
  \begin{pmatrix}
 f_1 \\
 f_2\\
 \vdots\\
 f_{N-1}\\
 f_N
 \end{pmatrix}.
 \label{zhang-scheme2}
\end{equation}
}
We refer to the tridiagonal $\frac{1}{6}(1,4,1)$ matrix as a weighting matrix. 
For the one-dimensional scalar conservation laws with periodic boundary conditions on  $[0,1]$:
\begin{equation}
u_t+f(u)_x  =  0,
\quad u(x,0)  =  u_0(x),
\label{1dconvection}
\end{equation}
the semi-discrete fourth order compact finite difference scheme can be written as
\begin{equation}
\label{zhang-scheme1}
 \frac{d\bar{u}_i}{dt}=-\frac{1}{2\Delta x}[f(u_{i+1})-f(u_{i-1})],
\end{equation}
where 
$\bar{u}_i$ is defined as $\bar{u}_i=\frac16 (u_{i-1}+4 u_i +u_{i+1}).
$
Let $\lambda=\frac{\Delta t}{\Delta x}$, then \eqref{zhang-scheme1} with the forward Euler time discretization  becomes
\begin{equation}
\bar{u}_i^{n+1}=\bar{u}_i^n-\frac{1}{2}\lambda[f(u^n_{i+1})-f(u^n_{i-1})].
\label{Euler2}
\end{equation}
The following weak monotonicity holds under the CFL $\lambda \max_u|f'(u)|\leq \frac13$:
\begin{align*}
  \bar{u}^{n+1}_i& =\frac16 (u^n_{i-1}+4 u^n_i + u^n_{i+1})-\frac12\lambda[f(u^n_{i+1})-f(u^n_{i-1})]\\
  &= \frac16[u_{i-1}+3\lambda f(u^n_{i-1})]+\frac16[u^n_{i+1}-3\lambda f(u^n_{i+1})]+\frac46 u^n_i\\
  &= H(u^n_{i-1},u^n_i, u^n_{i+1})=H(\uparrow, \uparrow, \uparrow),
\end{align*}
where $\uparrow$ denotes that the partial derivative with respect to the corresponding argument is non-negative. 
Therefore $m\leq u^n_{i}\leq M$ implies $m=H(m,m,m)\leq \bar u^{n+1}_{i}\leq H(M,M,M)=M,$ thus
\begin{equation}
 \label{weakmonotonicity}
 m\leq \frac16 (u^{n+1}_{i-1}+4 u^{n+1}_{i}+ u^{n+1}_{i+1})\leq M.
\end{equation}
If there is any overshoot or undershoot, i.e., $u^{n+1}_{i}>M$ or $u^{n+1}_{i}<m$ for some $i$, then 
\eqref{weakmonotonicity} implies that a local limiting process can eliminate the overshoot or undershoot.
Here we consider the special case $m=0$ to demonstrate the basic idea of this limiter, and for simplicity we ignore the time step index $n+1$. 
In Section \ref{sec-1d}
we will show that $\frac16 (u_{i-1}+4u_{i}+u_{i+1})\geq 0,\forall i$ implies the following two facts:
\begin{enumerate}
 \item $\max\{ u_{i-1}, u_i, u_{i+1}\}\geq 0;$
 \item If $u_i<0$, then $\frac12 (u_{i-1})_+ +\frac12 (u_{i+1})_+\geq -u_i> 0$, where $(u)_+=\max\{ u, 0\}.$
\end{enumerate}
By the two facts above, when $u_{i}<0$, then the following three-point stencil limiting process can enforce positivity without changing $\sum_i u_i$:
\begin{eqnarray*}
&& v_{i-1}=u_{i-1}+\frac{(u_{i-1})_+}{(u_{i-1})_+ +(u_{i+1})_+}u_i;\quad 
 v_{i+1}= u_{i+1}+\frac{(u_{i+1})_+}{(u_{i-1})_+ +(u_{i+1})_+}u_i, \\
&& \text{replace}\quad u_{i-1},\, u_i,\, u_{i+1}  \quad\text{by} \quad v_{i-1},\, 0, \, v_{i+1}\quad  \textrm{respectively}.
\end{eqnarray*}

In Section \ref{sec-limiter}, we will show that such a simple limiter can enforce the bounds of $u_i$ without destroying accuracy and conservation.
Thus with SSP high order time discretizations, the fourth order compact finite difference scheme solving \eqref{1dconvection}
can be rendered bound-preserving by this limiter. 
Moreover, in this paper we will show that such a weak monotonicity and the limiter can be easily extended to more general and practical cases including
two-dimensional problems, convection diffusion problems, inflow-outflow boundary conditions, higher order accurate compact finite difference approximations,
compact finite difference schemes with a total-variation-bounded (TVB) limiter \cite{cockburn1994nonlinearly}.
However, the extension to non-uniform grids is highly nontrivial thus will not be discussed. In this paper, we only focus on uniform grids.  
\subsection{The weak monotonicity for diffusion problems}

Although the weak monotonicity holds for arbitrarily high order  finite volume type schemes solving the convection equation \eqref{1dconvection}, 
it no longer holds for a conventional high order linear finite volume scheme or DG scheme even for the simplest heat equation, see the appendix in \cite{zhang2017positivity}.
Toward satisfying the weak monotonicity for the diffusion operator,  an unconventional high order finite volume scheme was constructed in 
\cite{zhang2012maximum2}. Second order accurate DG schemes usually satisfies  the weak monotonicity for the diffusion operator on general meshes \cite{zhang2013maximum}. 
The only previously known high order linear scheme in the literature satisfying the weak monotonicity for scalar diffusion problems is
the third order direct DG (DDG) method with special parameters  \cite{chen2016third},
which is a generalized version of interior penalty DG method.
On the other hand, arbitrarily high order nonlinear positivity-preserving DG schemes for diffusion problems were constructed in
\cite{zhang2017positivity, sun2018discontinuous, Srinivasan2018positivity}.

In this paper we will show that the fourth order accurate compact finite difference  and a few higher order accurate ones are also weakly monotone,
which is another class of linear high order schemes satisfying the weak monotonicity for diffusion problems.

It is straightforward to verify that the backward Euler or Crank-Nicolson method with the fourth order compact finite difference methods
 satisfies a maximum principle for the heat equation but it can be used 
be as a bound-preserving scheme only for linear problems. The method is this paper is explicit thus can be easily applied to nonlinear problems.
It is difficult to generalize the maximum principle to an implicit scheme.
Regarding  positivity-preserving implicit schemes, see \cite{qin2018implicit} for 
a study on weak monotonicity in implicit schemes solving convection equations. See also
\cite{hu2018asymptotic} for a second order accurate implicit and explicit  time discretization for the BGK equation.

\subsection{Contributions and organization of the paper}
Although high order compact finite difference methods have been extensively studied in the literature, e.g., 
\cite{lele1992compact, carpenter1993stability, cockburn1994nonlinearly,tolstykh1994high,  spotz1995high, tolstykh1998performance},
this is the first time that the weak monotonicity in compact finite difference approximations is discussed.
This is also the first time a weak monotonicity property is established for a high order accurate finite difference type scheme. 
The weak monotonicity property suggests it is possible to locally post process the numerical solution without losing conservation by a simple limiter 
to enforce global bounds. Moreover, this approach allows an easy justification of high order accuracy of the constructed bound-preserving scheme. 
 
For extensions to two-dimensional problems, convection diffusion problems, and sixth order and eighth order accurate schemes, 
the discussion about the weak monotonicity in general becomes more complicated since the weighting matrix 
may become a five-diagonal matrix instead of the tridiagonal $\frac16(1,4,1)$ matrix in \eqref{zhang-scheme2}. 
Nonetheless, we demonstrate that the same simple three-point stencil limiter can still be used to enforce bounds because we can
factor the more complicated weighting matrix as a product of a few of tridiagonal $\frac{1}{c+2}(1,c,1)$ matrices with $c\geq 2$.

 The paper is organized as follows: in Section \ref{sec-1d} we demonstrate the main idea for the fourth order accurate scheme 
solving one-dimensional problems with periodic boundary conditions.
 Two-dimensional extensions  are discussed in in Section \ref{sec-2d}. 
Section \ref{sec-highorder} is the extension to 
higher order accurate schemes. Inflow-outflow boundary conditions and Dirichlet boundary conditions are considered in Section \ref{sec-bc}. Numerical tests are given in 
Section \ref{sec-test}.
Section \ref{sec-remark} consists of concluding remarks. 

\section{A fourth order accurate scheme for one-dimensional problems}
\label{sec-1d}
In this section we first show the fourth order compact finite difference  with forward Euler time discretization
satisfies the weak monotonicity. 
Then we discuss how to design a simple limiter to enforce the bounds of point values.  
To eliminate the oscillations, a total variation bounded (TVB) limiter can be used. 
We also show that the TVB limiter does not affect the bound-preserving property of $\bar{u}_i$, 
thus it can be combined with the bound-preserving limiter to ensure the bound-preserving and non-oscillatory solutions for shocks. 
High order time discretizations will be discussed in Section \ref{sec-highordertime}.

\subsection{One-dimensional convection problems}
\label{sec-1dconvection}
Consider a periodic function $f(x)$ on the interval $[0,1]$.
Let $x_i=\frac{i}{N}$ $(i=1,\cdots, N)$ be the uniform grid points on the interval $[0,1]$. 
Let $\mathbf{f}$ be a column vector with numbers $f_1, f_2, \cdots, f_N$ as entries, where $f_i=f(x_i)$.
Let $W_1$, $W_2$, $D_x$ and $D_{xx}$ denote four linear operators as follows:
  {\setlength\arraycolsep{2pt} \[W_1\mathbf{f}=\frac16\begin{pmatrix}
              4 & 1 & & & 1\\
              1 & 4 & 1 & &  \\              
              &  \ddots & \ddots & \ddots & \\
              &  & 1 & 4 &1 \\
              1  & & & 1 & 4 \\
             \end{pmatrix}
             \begin{pmatrix}
 f_1 \\
 f_2\\
 \vdots\\
 f_{N-1}\\
 f_N
 \end{pmatrix}, D_x\mathbf{f} =
 \frac{1}{2}\begin{pmatrix}
              0 & 1 & & &  -1\\
              -1 & 0 & 1 & &  \\              
              &  \ddots & \ddots & \ddots & \\
              &  & -1 & 0 & 1 \\
              1  & & & -1 & 0 \\
             \end{pmatrix} 
  \begin{pmatrix}
 f_1 \\
 f_2\\
 \vdots\\
 f_{N-1}\\
 f_N
 \end{pmatrix},\]
} 

  {\setlength\arraycolsep{2pt} 
 \[W_2\mathbf{f}=\frac{1}{12}\begin{pmatrix}
              10 & 1 & & & 1\\
              1 & 10 & 1 & &  \\              
              &  \ddots & \ddots & \ddots & \\
              &  & 1 & 10 &1 \\
              1  & & & 1 & 10 \\
             \end{pmatrix}
             \begin{pmatrix}
 f_1 \\
 f_2\\
 \vdots\\
 f_{N-1}\\
 f_N
 \end{pmatrix},
 D_{xx}\mathbf{f} =
 \begin{pmatrix}
              -2 & 1 & & &  1\\
              1 & -2 & 1 & &  \\              
              &  \ddots & \ddots & \ddots & \\
              &  & 1 & -2 & 1 \\
              1  & & & 1 & -2 \\
             \end{pmatrix}
  \begin{pmatrix}
 f_1 \\
 f_2\\
 \vdots\\
 f_{N-1}\\
 f_N
 \end{pmatrix}.
 \]}
The fourth order compact finite difference approximation to the first order derivative \eqref{zhang-scheme2}
with periodic assumption for $f(x)$  can be denoted as
$ W_1 \mathbf{f}'= \frac{1}{\Delta x}D_x \mathbf{f}.$
The fourth order compact finite difference approximation to  $f''(x)$ is
$ W_2 \mathbf{f}''= \frac{1}{\Delta x^2} D_{xx} \mathbf{f}.$
The fourth compact finite difference approximations can be explicitly written as 
\[ \mathbf{f}'=  \frac{1}{\Delta x} W_1^{-1} D_x \mathbf{f},\quad \mathbf{f}''=  \frac{1}{\Delta x^2} W_2^{-1} D_{xx} \mathbf{f},\]
where $W_1^{-1}$ and $W_2^{-1}$ are the inverse operators. 
For convenience, by abusing notations we let $W_1^{-1} f_i$ denote the $i$-th entry of the vector $W_1^{-1} \mathbf f$. 

Then the scheme \eqref{zhang-scheme1}  solving the scalar conservation laws \eqref{1dconvection} with periodic boundary conditions on the interval
$[0,1]$ can be written as
$W_1\frac{d}{dt} u_i=-\frac{1}{2\Delta x} [f(u_{i+1})-f(u_{i-1})],$
and
the scheme \eqref{Euler2} is equivalent to
$W_1 u^{n+1}_i=W_1 u^{n}_i-\frac12\lambda [f(u^n_{i+1})-f(u^n_{i-1})].$
As shown in Section \ref{sec-intro-monotonicity}, the scheme \eqref{Euler2} satisfies the weak monotonicity.
\begin{thm}
\label{thm-weakmono-convec}
 Under the CFL constraint $\frac{\Delta t}{\Delta x}\max_u|f'(u)|\leq \frac13,$if $u^n_i\in[m,M]$, then 
 $u^{n+1}$ computed by the scheme \eqref{Euler2} satisfies \eqref{weakmonotonicity}. 
\end{thm}
\subsection{A three-point stencil bound-preserving limiter}
\label{sec-limiter}
In this subsection, 
we consider a more general constraint than  \eqref{weakmonotonicity} 
and we will design a simple limiter to enforce bounds of point values based on it. 
Assume we are given a sequence of periodic point values $u_i$ $(i=1,\cdots, N)$ satisfying 
 \begin{equation}
 m\leq \frac{1}{c+2}( u_{i-1}+ cu_{i}+ u_{i+1})\leq M, \quad i=1,\cdots, N,\quad c\geq 2,
 \label{weakmonotonicity2}
 \end{equation}
 where $u_0:=u_N$, $u_{N+1}:=u_1$ and $c\geq 2$ is a constant.
We have the following results:
 \begin{lem}
 \label{firstlemma}
The constraint \eqref{weakmonotonicity2} implies the following for stencil $\{i-1, i, i+1\}$:
  \begin{enumerate}
  \item[(1)]  $\min\{u_{i-1}, u_i, u_{i+1}\}\leq M, \quad \max\{u_{i-1}, u_i, u_{i+1}\}\geq m.$
  \item[(2)]   If $u_i>M$, then $\frac{(u_i-M)_+}{(M-u_{i-1})_+ +(M-u_{i+1})_+}\leq \frac{1}{c}$.\\
  If $u_i<m$, then $\frac{(m-u_i)_+}{(u_{i-1}-m)_+ +(u_{i+1}-m)_+}\leq \frac{1}{c}.$\\
  Here the subscript $+$ denotes the positive part, i.e., $(a)_+=\max\{a,0\}.$
  \end{enumerate}
\end{lem}
\begin{rmk}
The first statement in Lemma \ref{firstlemma} states that there do not exist three consecutive overshoot points or three consecutive undershoot points. 
But it does not necessarily imply that at least one of three consecutive point values is in the bounds $[m,M]$.
For instance, consider the case for $c=4$ and $N$ is even, 
define $u_i\equiv 1.1$ for all odd $i$ and $u_i\equiv -0.1$ for all even $i$, then $\frac{1}{c+2}( u_{i-1}+ cu_{i}+ u_{i+1})\in[0,1]$ for all $i$ but none of the point values $u_i$ is in $[0,1]$.
\end{rmk}
\begin{rmk}
Lemma \ref{firstlemma} implies that if  $u_i$ is out of the range $[m, M]$, then
we can set $u_i\leftarrow m$ for undershoot (or $u_i\leftarrow M$ for overshoot) without changing the local sum $u_{i-1}+u_i+u_{i+1}$ 
by decreasing (or increasing) its neighbors $u_{i\pm1}$.
\end{rmk}
\begin{proof}

We only discuss the upper bound. The inequalities for the lower bound can be similarly proved.
First, if $u_{i-1}, u_i, u_{i+1}>M$ then $\frac{1}{c+2}(u_{i-1}+ cu_{i}+ u_{i+1})>M$ which is a contradiction to \eqref{weakmonotonicity2}. 
Second, \eqref{weakmonotonicity2} implies $u_{i-1}+cu_i+u_{i+1}\leq (c+2)M$, thus $c(u_i-M)\leq (M-u_{i-1})+(M-u_{i+1})\leq (M-u_{i-1})_+ +(M-u_{i+1})_+$. If $u_i>M$,
 we get $(M-u_{i-1})_+ +(M-u_{i+1})_+>0$. Moreover, $\frac{(u_i-M)_+}{(M-u_{i-1})_+ +(M-u_{i+1})_+}=\frac{u_i-M}{(M-u_{i-1})_+ +(M-u_{i+1})_+}\leq \frac{1}{c}$.
 \end{proof}

For simplicity, we first consider a limiter to enforce only the lower bound without destroying global conservation. 
For  $m=0$, this is a positivity-preserving limiter.

\begin{algorithm}
\caption{A limiter for periodic data $u_i$ to enforce the lower bound.}
\begin{algorithmic}
\REQUIRE The input $u_i$ satisfies $\bar u_i=\frac{1}{c+2} (u_{i-1}+c u_{i}+u_{i+1})\geq m, i=1,\cdots, N$, with $c\geq 2$. Let $u_0$, $u_{N+1}$ denote $u_N$, $u_1$ respectively. 
\ENSURE The output satisfies $v_i\geq m, i=1,\cdots, n$ and $\sum_{i=1}^N v_i=\sum_{i=1}^N u_i$.
\STATE First set $v_i=u_i$, $i=1,\cdots, N$.  Let $v_0$, $v_{N+1}$ denote $v_N$, $v_1$ respectively. 
\FOR{$i=1,\cdots, N$}
\IF{$u_i <m$}
\STATE $v_{i-1}\leftarrow  v_{i-1}-\frac{(u_{i-1}-m)_+}{(u_{i-1}-m)_+ +(u_{i+1}-m)_+}(m-u_i)_+$ 
\STATE $v_{i+1} \leftarrow v_{i+1}-\frac{(u_{i+1}-m)_+}{(u_{i-1}-m)_+ +(u_{i+1}-m)_+}(m-u_i)_+$
\STATE $v_i\leftarrow m$
\ENDIF
\ENDFOR
\end{algorithmic}
\label{alg_limiter}
\end{algorithm}

\begin{rmk}
 Even though a {\bf for} loop is used,  Algorithm \ref{alg_limiter} is a local operation to an undershoot point since only information of two immediate neighboring points of the undershoot point are needed.  Thus it is not a sweeping limiter. 
\end{rmk}

\begin{thm}\label{theorem-lowerbound}
The output of Algorithm \ref{alg_limiter}
satisfies $\sum\limits_{i=1}^N v_i=\sum\limits_{i=1}^N u_i$ and $v_i\geq m$. 
\end{thm}

\begin{proof}
First of all, notice that the algorithm only modifies the undershoot points and their immediate neighbors. 

Next we will show the output satisfies $v_i\geq m$ case by case:

\begin{itemize}
\item If $u_i<m$, the $i$-th step in {\bf for} loops sets $v_i=m$.
After the $(i+1)$-th step in {\bf for} loops, we still have $v_i=m$ because $(u_i-m)_+=0$.

\item If $u_i=m$, then $v_i=m$ in the final output because $(u_i-m)_+=0$.
\item If $u_i> m$, then limiter may decrease it if at least one of its neighbors $u_{i-1}$ and $u_{i+1}$ is below $m$:
\begin{eqnarray*}
 v_i &=& u_i- \frac{(u_{i}-m)_+(m-u_{i-1})_+}{(u_{i-2}-m)_+ +(u_{i}-m)_+}-\frac{(u_{i}-m)_+(m-u_{i+1})_+}{(u_{i}-m)_+ +(u_{i+2}-m)_+}\\
 &\geq& u_i-\frac{1}{c}(u_{i}-m)_+-\frac{1}{c}(u_{i}-m)_+>  m,
\end{eqnarray*}
where the inequalities are implied by  Lemma \ref{firstlemma}  and the fact $c\geq 2$.
\end{itemize}

Finally, we need to show the local sum $v_{i-1}+v_i+v_{i+1}$ is not changed during the $i$-th step if $u_i<m$. If $u_i<m$, then after $(i-1)$-th step we still have $v_i=u_i$ because $(u_i-m)_+=0$. Thus in the $i$-th step of {\bf for} loops, the point value at $x_i$ is increased by the amount $m-u_i$, and the point values at $x_{i-1}$ and $x_{i+1}$ are decreased by $\frac{(u_{i-1}-m)_+}{(u_{i-1}-m)_+ +(u_{i+1}-m)_+}(m-u_i)_++\frac{(u_{i+1}-m)_+}{(u_{i-1}-m)_+ +(u_{i+1}-m)_+}(m-u_i)_+=m-u_i$. So $v_{i-1}+v_i+v_{i+1}$ is not changed during the $i$-th step.
Therefore the limiter ensures the output $v_i\geq m$ without changing the global sum. 
\end{proof}

 The limiter described by Algorithm \ref{alg_limiter} is a local three-point stencil limiter in the sense that only undershoots and their neighbors will be modified, 
 which means the limiter has no influence on point values that are neither undershoots nor neighbors to undershoots.
 Obviously a similar procedure can be used to enforce only the upper bound. 
 However, to enforce both the lower bound and the upper bound, the discussion for this three-point stencil limiter is complicated for a saw-tooth profile in which 
 both neighbors of an overshoot point are undershoot points. Instead, we will use a different limiter for the saw-tooth profile. 
 To this end, we need to separate the point values $\{u_i, i=1,\cdots, N\}$ into two classes of subsets consisting of consecutive point values. 
 
 In the following discussion, a {\it set} refers to a set of consecutive point values $u_l, u_{l+1}, u_{l+2}, \cdots, u_{m-1}, u_m$.
For any set $S=\{u_l, u_{l+1}, \cdots, u_{m-1}, u_m\}$, we call the first point value $u_l$ and the last point value $u_m$ as {\it boundary points},
and call the other point values  ${u_{l+1}, \cdots, u_{m-1}}$ as  {\it interior points}.
A set of class I is defined as a set satisfying the following:
\begin{enumerate}
 \item It contains at least four point values.
 \item Both {\it boundary points} are in $[m,M]$ and all {\it interior points} are out of range. 
 \item It contains both undershoot and overshoot points.
\end{enumerate}
Notice that in a set of class I, at least one undershoot point is next to an overshoot point.
For given point values $u_i,i=1,\cdots, N$, 
suppose all the sets of class I are $S_1=\{u_{m_1},u_{m_1+1},\cdots, u_{n_1}\}$, $S_2=\{u_{m_2},\cdots, u_{n_2}\}$, $\cdots$, $S_K=\{u_{m_K},\cdots, u_{n_K}\}$, 
where $m_1<m_2<\cdots< u_{m_K}$.

A set of class II consists of point values between $S_i$ and $S_{i+1}$ and two boundary points $u_{n_i}$ and $u_{m_{i+1}}$. 
Namely they are
$T_{0}=\{u_1,u_2,\cdots,u_{m_1}\}$, $T_{1}=\{u_{n_1},\cdots,u_{m_2}\}$, $T_{2}=\{u_{n_2},\cdots,u_{m_3}\}$, $\cdots$, $T_{K}=\{u_{n_K},\cdots,u_N\}$.
For periodic data $u_i$, we can combine $T_{K}$ and $T_0$ to define $T_K=\{u_{n_K},\cdots,u_{N},u_1,\cdots,u_{m_1}\}$.

In the sets of class I, the undershoot and the overshoot are neighbors. 
In the sets of class II,  the undershoot and the overshoot are separated, i.e., an overshoot is not next to any undershoot.
We remark that the sets of class I are hardly encountered in the numerical tests but we include them in the discussion for the sake of completeness.
When there are no sets of class I, all point values form a single set of class II. 
We will use the same procedure as in Algorithm \ref{alg_limiter}  for $T_i$ and a different limiter for $S_i$ 
to  enforce both the lower bound and the upper bound.

\begin{algorithm}
\caption{A bound-preserving limiter for periodic data $u_i$ satisfying $\bar u_i\in [m, M]$}
\begin{algorithmic}[1]
\REQUIRE the input $u_i$ satisfies $\bar u_i=\frac{1}{c+2}( u_{i-1}+c u_{i}+ u_{i+1})\in [m, M]$, $c\geq 2$. Let $u_0$, $u_{N+1}$ denote $u_N$, $u_1$ respectively. 
\ENSURE the output satisfies $v_i\in[m, M], i=1,\cdots, N$ and $\sum_{i=1}^N v_i=\sum_{i=1}^N u_i$.
\STATE {\bf Step 0}: First set $v_i=u_i$, $i=1,\cdots, N$. Let $v_0$, $v_{N+1}$ denote $v_N$, $v_1$ respectively. 
\STATE {\bf Step I}: Find all the sets of class I $S_1,\cdots,S_K$ (all local saw-tooth profiles) and all the sets of class II $T_{1},\cdots, T_{K}$.
\STATE {\bf Step II}: For each $T_j$ $(j=1,\cdots, K)$, the same limiter as in  Algorithm \ref{alg_limiter} (but for both upper bound and lower bound) is used: 
\FOR{all index $i$ in $T_j$}
\IF{$u_i <m$}
\STATE $v_{i-1}\leftarrow  v_{i-1}-\frac{(u_{i-1}-m)_+}{(u_{i-1}-m)_+ +(u_{i+1}-m)_+}(m-u_i)_+$
\STATE $v_{i+1} \leftarrow v_{i+1}-\frac{(u_{i+1}-m)_+}{(u_{i-1}-m)_+ +(u_{i+1}-m)_+}(m-u_i)_+$
\STATE $v_i\leftarrow m$
\ENDIF
\IF{$u_i >M$}
\STATE $v_{i-1}\leftarrow  v_{i-1}+\frac{(M-u_{i-1})_+}{(M-u_{i-1})_+ +(M-u_{i+1})_+}(u_i-M)_+$
\STATE $v_{i+1} \leftarrow v_{i+1}+\frac{(M-u_{i+1})_+}{(M-u_{i-1})_+ +(M-u_{i+1})_+}(u_i-M)_+$ 
\STATE $v_i\leftarrow M$
\ENDIF
\ENDFOR
\STATE {\bf Step III}: for each saw-tooth profile $S_j=\{u_{m_j},\cdots, u_{n_j}\}$ $(j=1,\cdots, K)$,
 let $N_0$ and $N_1$ be the numbers of undershoot and overshoot points in $S_j$ respectively. 
 \STATE Set $U_j=\sum_{i=m_j}^{n_j }v_i$.
\FOR{$i=m_j+1,\cdots, n_j-1$}
\IF{$u_i >M$}
\STATE $v_i\leftarrow M$.
\ENDIF
\IF{$u_i <m$}
\STATE $v_i\leftarrow m$.
\ENDIF
\ENDFOR

\STATE Set $V_j=N_1M +N_0 m+v_{m_j}+v_{n_j}$.
\STATE Set $A_j=v_{m_j}+v_{n_j}+N_1M-(N_1+2)m$, $B_j=(N_0+2) M-v_{m_j}-v_{n_j}-N_0 m$.
\IF {$V_j-U_j>0$}
\FOR{$i=m_j,\cdots, n_j$}
\STATE $v_i\leftarrow v_i-\frac{v_i-m}{A_j}(V_j-U_j)$
\ENDFOR
\ELSE
\FOR{$i=m_j,\cdots, n_j$}
\STATE $v_i \leftarrow v_i+\frac{M-v_i}{B_j}(U_j-V_j)$
\ENDFOR
\ENDIF
\end{algorithmic}
\label{alg_limiter2}
\end{algorithm}

\begin{thm}\label{1dconvectionthm}
Assume periodic data $u_i (i=1,\cdots, N)$ satisfies
$\bar u_i=\frac{1}{c+2}( u_{i-1}+c u_{i}+ u_{i+1})\in [m, M]$, $c\geq 2$ for all $i=1,\cdots, N$ with $u_0:=u_N$ and $u_{N+1}:=u_1$,
then the output  of Algorithm \ref{alg_limiter2} 
satisfies $\sum_{i=1}^N v_i=\sum_{i=1}^N u_i$ and $v_i\in [m,M],$ $\forall i$. 
\end{thm}

\begin{proof}
 
 First we show the output $v_i\in [m,M]$. Consider {\bf Step II}, which only modifies the undershoot and overshoot points and their immediate neighbors.
Notice that the operation described by lines 6-8 will not increase the point value of neighbors to an undershoot point thus it will not create new overshoots. Similarly, the operation described by lines 11-13 will not create new undershoots. In other words, no new undershoots (or overshoots) will be created when eliminating overshoots (or undershoots) in {\bf Step II}.

Each interior point $u_i$ in any $T_j$ belongs to one of the following four cases:
\begin{enumerate}
\item $u_i\leq m$ or $u_i\geq M$.
\item $m<u_i<M$ and $u_{i-1}, u_{i+1}\leq M$.
\item $m<u_i<M$ and $u_{i-1}, u_{i+1}\geq m$.
\item $m<u_i<M$ and $u_{i-1}>M, u_{i+1}<m$ (or $u_{i+1}>M, u_{i-1}<m$).
\end{enumerate}
We want to show $v_i\in[m, M]$ after {\bf Step II}.
For the first three cases, by the same arguments as in the proof of Theorem \ref{theorem-lowerbound}, we can easily show that the output point values are in the range $[m,M]$. 
For case (1), after {\bf Step II}, if $u_i\leq m$ then $v_i=m$; if $u_i\geq M$ then $v_i=M$. 
For case (2), $v_i\neq u_i$ only if at least one of $u_{i-1}$ and $u_{i+1}$ is an undershoot. If so, then \begin{eqnarray*}
 v_i &=& u_i- \frac{(u_{i}-m)_+(m-u_{i-1})_+}{(u_{i-2}-m)_+ +(u_{i}-m)_+}-\frac{(u_{i}-m)_+(m-u_{i+1})_+}{(u_{i}-m)_+ +(u_{i+2}-m)_+}\\
 &\geq& u_i-\frac{1}{c}(u_{i}-m)_+-\frac{1}{c}(u_{i}-m)_+>  m.
\end{eqnarray*}
Similarly, for case (3), $v_i\neq u_i$ only if at least one of $u_{i-1}$ and $u_{i+1}$ is an overshoot, and we can show $v_i<M$.

Notice that case (2) and case (3) are not exclusive to each other, which however does not affect the discussion here. When case (2) and case (3) overlap, we have $u_i, u_{i-1}, u_{i+1}\in[m, M]$ thus $v_i=u_i\in [m, M]$ after {\bf Step II}.

For case (4), without loss of generality, we consider the case when $u_{i+1}>M, u_i\in[m,M], u_{i-1}<m$, and 
we need to show that the output $v_i\in[m, M]$. 
By Lemma \ref{firstlemma}, we know that 
Algorithm \ref{alg_limiter2} will  decrease the value at $x_i$ by
at most $\frac{1}{c}(u_{i}-m)$ to eliminate the undershoot at $x_{i-1}$ then 
increase the point value at $x_i$ by at most $\frac{1}{c}(M-u_{i})$ to eliminate the overshoot at $x_{i+1}$. So after {\bf Step II}, 
\begin{eqnarray*}
v_i&\leq& u_i+\frac{1}{c}(M-u_{i})\leq M \quad (\text{because}\quad c\geq 2, u_i<M);\\
v_i&\geq& u_i-\frac{1}{c}(u_{i}-m)\geq m \quad (\text{because}\quad c\geq 2, u_i>m).
\end{eqnarray*}
Thus we have $v_i\in[m,M]$ after {\bf Step II}. 
By the same arguments as in the proof of Theorem \ref{theorem-lowerbound}, we can also easily show the boundary points are in the range $[m, M]$  after {\bf Step II}.
It is straightforward to verify that $\sum_{i=1}^N v_i=\sum_{i=1}^N u_i$ after {\bf Step II} because the operations described by lines 6-8 and lines 11-13 do not change the local sum $v_{i-1}+v_i+v_{i+1}$.

Next we discuss {\bf Step III} in Algorithm \ref{alg_limiter2}. 
 Let $\bar N=2+N_0+N_1=n_j-m_j+1$ be the cardinality of $S_j=\{u_{m_j},\cdots, u_{n_j}\}$.
 
We need to show that the average value in each saw-tooth profile $S_j$
 is in the range $[m, M]$ after {\bf Step II} before {\bf Step III}. Otherwise it is impossible to 
enforce the bounds in $S_j$ without changing the sum in $S_j$. In other words, we need to show $\bar N m \leq U_j=\sum_{v_i\in S_j}v_i\leq \bar N M$.
We will prove the claim by conceptually applying the upper or lower bound limiter Algorithm \ref{alg_limiter} to $S_j$. Consider a boundary point of $S_j$, e.g., $u_{m_j}\in[m,M]$, then during  {\bf Step II} the point value at $x_{m_j}$ can be unchanged, moved down at most $\frac{1}{c}(u_{m_j}-m)$ or moved up at most $\frac{1}{c}(M-u_{m_j})$. 
We first show the average value in $S_j$ after {\bf Step II} is not below $m$:
\begin{itemize}
 \item[(a)] Assume both boundary point values of $S_j$ are unchanged during {\bf Step II}. If applying Algorithm \ref{alg_limiter} to $S_j$ after {\bf Step II}, by the proof of
 Theorem \ref{theorem-lowerbound}, we know that the output values would be greater than or equal to $m$ with the same sum, which implies that 
 $\sum_{v_i\in S_j}v_i\geq \bar N m$. 
\item[(b)] If a boundary point value of $S_j$ is increased during {\bf Step II}, the same discussion as in (a) still holds because an increased boundary value does not affect the discussion for the lower bound.
\item[(c)] If a boundary point value $v_{m_j}$ of $S_j$ is decreased during {\bf Step II}, then with the fact that it is decreased by at most the amount $\frac{1}{c}(u_{m_j}-m)$, the same discussion as in (a) still holds. 
\end{itemize}

 Similarly if applying the upper bound limiter similar to Algorithm \ref{alg_limiter} to $S_j$ after {\bf Step II},
 then by the similar arguments as above,  the output values would be less than or equal to $M$ with the same sum,
 which implies $\sum_{v_i\in S_j}v_i\leq \bar N M$.
 
 Now we can show the output $v_i\in [m,M]$ for each $S_j$ after {\bf Step III}:
  \begin{enumerate}
  \item   Assume $V_j=N_1M +N_0 m+v_{m_j}+v_{n_j}>U_j$ before the {\bf for} loops in {\bf Step III}. Then after {\bf Step III}:
  if $u_i<m$ we get $v_i=m$; 
  if $u_i\geq m$ we have  
  \begin{eqnarray*}
    M\geq &v_i&-\frac{v_i-m}{A_j}(V_j-U_j)\nonumber\\    
    = &v_i&-\frac{v_i-m}{v_{m_j}+v_{n_j}+N_1M-(N_1+2)m}(v_{m_j}+v_{n_j}+N_1M+N_0m-U_j)\nonumber\\
   \geq &v_i& -\frac{v_i-m}{v_{m_j}+v_{n_j}+N_1M-(N_1+2)m}(v_{m_j}+v_{n_j}+N_1M+N_0m-\bar Nm)\nonumber\\
   = &v_i&-(v_i-m)=m.\nonumber
     \end{eqnarray*}
  \item   Assume $V_j=N_1M +N_0 m+v_{m_j}+v_{n_j}\leq U_j$ before the {\bf for} loops in {\bf Step III}. Then after {\bf Step III}: 
  if $u_i>M$ we get $v_i=M$; 
  if $u_i\geq M$ we have  
   \begin{eqnarray*}
    m\leq &v_i&+\frac{M-v_i}{B_j}(U_j-V_j)\nonumber\\
    = &v_i&+\frac{M-v_i}{(N_0+2)M-v_{m_j}-v_{n_j}-N_0m}(U_j-v_{m_j}-v_{n_j}-N_1M-N_0m)\nonumber\\
   \leq &v_i&+\frac{M-v_i}{(N_0+2)M-v_{m_j}-v_{n_j}-N_0m}(\bar N M-v_{m_j}-v_{n_j}-N_1M-N_0m)\nonumber\\
   =&v_i&+(M-v_i)=M.\nonumber
      \end{eqnarray*}
 \end{enumerate}
Thus we have shown all the final output values are in the range $[m,M]$.  

Finally it is straightforward to verify that $\sum_{i=1}^N v_i=\sum_{i=1}^N u_i$.
\end{proof}

The limiters described in Algorithm \ref{alg_limiter} and Algorithm \ref{alg_limiter2} are high order accurate limiters in the following sense.
Assume $u_i (i=1,\cdots, N)$ are high order accurate approximations to point values of a very smooth function $u(x)\in[m, M]$, i.e., $u_i-u(x_i)=\mathcal O(\Delta x^k)$.
For fine enough uniform mesh, the global maximum points are well separated from the global minimum points in $\{u_i, i=1, \cdots, N\}$.
In other words, there is no saw-tooth profile in $\{u_i, i=1, \cdots, N\}$. Thus Algorithm \ref{alg_limiter2} reduces to the three-point stencil limiter for 
smooth profiles on fine resolved meshes. Under these assumptions, the amount which limiter increases/decreases each point value is at most $(u_i-M)_+$ and  $(m-u_i)_+$.
If $(u_i-M)_+>0$, which means $u_i>M\geq u(x_i)$, we have $(u_i-M)_+=O(\Delta x^k)$ because $(u_i-M)_+<u_i-u(x_i)=O(\Delta x^k)$.
Similarly, we get $(m-u_i)_+=O(\Delta x^k)$. Therefore, for point values $u_i$ approximating a smooth function, the limiter 
changes $u_i$ by $O(\Delta x^k)$.

\subsection{A TVB limiter}
The scheme \eqref{Euler2} can be written into a conservation form:
\begin{equation}\label{conservationform}
 \bar{u}_i^{n+1}=\bar{u}_i^n-\frac{\Delta t}{\Delta x}(\hat{f}_{i+\frac{1}{2}}-\hat{f}_{i-\frac{1}{2}}),
\end{equation}
which is suitable for shock calculations and involves a numerical flux 
\begin{equation}\label{flux}
 \hat{f}_{i+\frac{1}{2}}=\frac{1}{2}(f(u^n_{i+1})+f(u^n_i)).
\end{equation}
To achieve nonlinear stability and eliminate oscillations for shocks, a TVB (total variation bounded in the means) limiter was introduced 
for the scheme \eqref{conservationform} in \cite{cockburn1994nonlinearly}.
In this subsection we will show that the bound-preserving property of $\bar{u}_i$ \eqref{weakmonotonicity} still holds for the scheme \eqref{conservationform} with the TVB limiter in  \cite{cockburn1994nonlinearly}.
 Thus we can use both the TVB limiter and the bound-preserving limiter in Algorithm \eqref{alg_limiter2} at the same time.
 
 The compact finite difference scheme with the limiter in \cite{cockburn1994nonlinearly} is 
 \begin{eqnarray}\label{TVBMscheme}
  \bar{u}_i^{n+1}=\bar{u}_i^{n}-\frac{\Delta t}{\Delta x}(\hat{f}^{(m)}_{i+\frac12}-\hat{f}^{(m)}_{i-\frac12}),
 \end{eqnarray}
where the numerical flux $\hat{f}^{(m)}_{i+\frac12}$ is the modified flux approximating \eqref{flux}.

First we write 
$f(u)=f^+(u)+f^-(u)$ with the requirement that 
$\frac{\partial f^+(u)}{\partial u}\geq 0$,  and $\frac{\partial f^-(u)}{\partial u}\leq 0.$
The simplest such splitting is the Lax-Friedrichs splitting
$f^{\pm}(u)=\frac12(f(u)\pm\alpha u), \alpha=\max\limits_{u\in[m, M]}|f'(u)|$. Then we  write the flux $\hat{f}_{i+\frac{1}{2}}$ as 
$
 \hat{f}_{i+\frac12}=\hat{f}^+_{i+\frac12}+\hat{f}^-_{i+\frac12},
$
where $\hat{f}^{\pm}_{i+\frac12}$ are obtained by adding superscripts $\pm$ in \eqref{flux}.
Next we define 
\begin{equation*}
 d\hat{f}^+_{i+\frac12}=\hat{f}^+_{i+\frac12}-f^+(\bar{u}_i), \quad d\hat{f}^-_{i+\frac12}=f^-(\bar{u}_{i+1})-\hat{f}^-_{i+\frac12}.
\end{equation*}
Here $d\hat{f}^{\pm}_{i+\frac12}$ are the differences between the numerical fluxes $\hat{f}^{\pm}_{i+\frac12}$ and the first-order, upwind fluxes $f^+(\bar{u}_i)$ 
and $f^-(\bar{u}_{i+1})$. The limiting is defined by 
\[
 d\hat{f}^{+(m)}_{i+\frac{1}{2}}=\tilde{m}(d\hat{f}^+_{i+\frac12}, \Delta^+f^+(\bar{u}_i), \Delta^+f^+(\bar{u}_{i-1})),\quad 
 d\hat{f}^{-(m)}_{i+\frac{1}{2}}=\tilde{m}(d\hat{f}^-_{i+\frac12}, \Delta^+f^-(\bar{u}_i), \Delta^+f^-(\bar{u}_{i+1})),
\]
where $\Delta^+v_i\equiv v_{i+1}-v_i$ is the usual forward difference operator, and the modified $minmod$ function $\tilde{m}$ is defined by

\begin{equation}
\label{minmod}
\tilde{m}(a_1,\dots,a_k)=\left\{\begin{array}{ll}
 a_1, & \textrm{if $|a_1|\leq p\Delta x^2$,}\\
 m(a_1,\dots,a_k), & \textrm{otherwise,}
\end{array}\right.
\end{equation}
where $p$ is a positive constant independent of $\Delta x$ and $m$ is the  $minmod$ function
\begin{equation*}
m(a_1,\dots,a_k)=\left\{\begin{array}{ll}
 s \min_{1\leq i\leq k}|a_i|, & \textrm{if $sign(a_1)=\cdots=sign(a_k)=s$,}\\
 0, & \textrm{otherwise.}
\end{array}\right.
\end{equation*}\
The limited numerical flux is then defined by
$
 \hat{f}^{+(m)}_{i+\frac12}=f^+(\bar{u}_i)+d\hat{f}^{+(m)}_{i+\frac{1}{2}}, \quad \hat{f}^{-(m)}_{i+\frac12}=f^-(\bar{u}_{i+1})-d\hat{f}^{-(m)}_{i+\frac{1}{2}},
$and 
$
 \hat{f}^{(m)}_{i+\frac12}=\hat{f}^{+(m)}_{i+\frac12}+\hat{f}^{-(m)}_{i+\frac12}.
$
The following result was proved in \cite{cockburn1994nonlinearly}:
\begin{lem}
  For any $n$ and $\Delta t$ such that $0\leq n\Delta t\leq T$, scheme \eqref{TVBMscheme} is TVBM (total variation bounded in the means):
$
  TV(\bar{u}^n)=\sum_{i}|\bar{u}_{i+1}^n-\bar{u}_i^n|\leq C,
 $
where C is independent of $\Delta t$, under the CFL condition 
$
 \max_{u}(\frac{\partial}{\partial u}f^{+}(u)-\frac{\partial}{\partial u}f^{-}(u))\frac{\Delta t}{\Delta x}\leq \frac12.
$
 
\end{lem}

 Next we show that the TVB scheme still satisfies \eqref{weakmonotonicity}.
 
\begin{thm}
 If $u^n_i\in[m, M]$, then under a suitable CFL condition, the TVB scheme \eqref{TVBMscheme} satisfies
$m\leq \frac{1}{6}(u^{n+1}_{i-1}+4u^{n+1}_{i}+u^{n+1}_{i+1})\leq M.$
 \end{thm}

\begin{proof}
Let $\lambda=\frac{\Delta t}{\Delta x}$, then we have
  {\setlength\arraycolsep{2pt}
\begin{eqnarray*} 
 \bar{u}^{n+1}_{i} 
 & = & \bar{u}^n_i-\lambda (\hat{f}^{(m)}_{i+\frac12}-\hat{f}^{(m)}_{i-\frac12}) \\
 & = & \frac14(\bar{u}^n_i-4\lambda\hat{f}_{i+\frac12}^{+(m)})+\frac14(\bar{u}^n_i-4\lambda\hat{f}_{i+\frac12}^{-(m)})
 +\frac14(\bar{u}^n_i+4\lambda\hat{f}_{i-\frac12}^{+(m)})+\frac14(\bar{u}^n_i+4\lambda\hat{f}_{i-\frac12}^{-(m)}) \label{4terms}.
 \end{eqnarray*}
 }
  We will show $\bar{u}^{n+1}_{i}\in[m,M]$ 
  by proving that the four terms satisfy
  \begin{eqnarray*}
  & \bar{u}^n_i-4\lambda\hat{f}_{i+\frac12}^{+(m)}\in[m-4\lambda f^+(m), M-4\lambda f^+(M)],\\
  &  \bar{u}_i-4\lambda\hat{f}_{i+\frac12}^{-(m)}\in[m-4\lambda f^-(m), M-4\lambda f^-(M)],\\
  & \bar{u}^n_i+4\lambda\hat{f}_{i-\frac12}^{+(m)}\in[m+4\lambda f^+(m), M+4\lambda f^+(M)],\\
  &  \bar{u}_i+4\lambda\hat{f}_{i-\frac12}^{-(m)}\in[m+4\lambda f^-(m), M+4\lambda f^-(M)],
   \end{eqnarray*}
    under the CFL condition
    \begin{equation}
     \label{CFLtvb}
\lambda \max_u |f^{(\pm)}(u)|\leq 
 \frac{1}{12}.    \end{equation}
   We only discuss the first term since the proof for the rest is similar. 
    We notice that $u-4\lambda f^+(u)$ and $u-12\lambda f^+(u)$ are monotonically increasing functions of $u$ under the CFL constraint \eqref{CFLtvb},
   thus $u\in[m, M]$ implies $u-4\lambda f^+(u)\in [m-4\lambda f^+(m), M-4\lambda f^+(M)]$ and $ u-12\lambda f^+(u)\in [m-12\lambda f^+(m), M-12\lambda f^+(M)]$. 
   For convenience, we drop the time step $n$, then we have
\[
  \bar{u}_i-4\lambda\hat{f}_{i+\frac12}^{+(m)}  =  \bar{u}_i-4\lambda(f^+(\bar{u}_i)+d\hat{f}^{+(m)}_{i+\frac12}),
 \]
where the value of $d\hat{f}^{+(m)}_{i+\frac12}$ has four possibilities:
  \begin{enumerate}
   \item If $d\hat{f}^{+(m)}_{i+\frac12}=0$, then $$\bar{u}_i-4\lambda\hat{f}_{i+\frac12}^{+(m)}=\bar{u}_i-4\lambda f^+(\bar{u}_i)\in [m-4\lambda f^+(m), M-4\lambda f^+(M)].$$
   \item If $d\hat{f}^{+(m)}_{i+\frac12}=d\hat{f}^+_{i+\frac12}$, then we get 
   \begin{align*}
   \bar{u}_i-4\lambda\hat{f}_{i+\frac12}^{+(m)}&=\frac{1}{6}(u_{i-1}+4u_{i}+u_{i+1})-4\lambda\frac{f^+(u_i)+f^+(u_{i+1})}{2}\\
   &=\frac{1}{6}u_{i-1}+\frac23(u_{i}-3\lambda f^+(u_i))+\frac{1}{6}(u_{i+1}-12\lambda f^+(u_{i+1})).
   \end{align*}
   By the monotonicity of the function $u-12\lambda f^+(u)$ and $u-3\lambda f^+(u)$, we have
\begin{eqnarray*}
&u_{i}-3\lambda f^+(u_i)\in [m-3\lambda f^+(m), M-3\lambda f^+(M)],\\
 &  u_{i+1}-12\lambda f^+(u_{i+1})\in [m-12\lambda f^+(m), M-12\lambda f^+(M)],
  \end{eqnarray*} 
   which imply $\bar{u}_i-4\lambda\hat{f}_{i+\frac12}^{+(m)}\in [m-4\lambda f^+(m), M-4\lambda f^+(M)].$
   \item If $d\hat{f}^{+(m)}_{i+\frac12}=\Delta^+f^+(\bar{u}_i)$, $\bar{u}_i-4\lambda\hat{f}_{i+\frac12}^{+(m)}=\bar{u}_i-4\lambda f^+(\bar{u}_{i+1})$. 
   If $\Delta^+f^+(\bar{u}_i)>0$, $\bar{u}_i-4\lambda f^+(\bar{u}_{i+1})<\bar{u}_i-4\lambda f^+(\bar{u}_i)\leq M-4\lambda f^+(M)$, which implies the upper bound holds. 
   Due to the definition of the $minmod$ function, we can get $0<\Delta^+f^+(\bar{u}_i)<d\hat{f}^+_{i+\frac12}$. 
   Thus, $\hat{f}^+_{i+\frac12}=\frac{f^+(u_i)+f^+(u_{i+1})}{2}=f^+(\bar{u}_i)+d\hat{f}^+_{i+\frac12}>f^+(\bar{u}_i)+\Delta^+f^+(\bar{u}_i)=f^+(\bar{u}_{i+1})$. 
   Then, $\bar{u}_i-4\lambda f^+(\bar{u}_{i+1})>\bar{u}_i-4\lambda\frac{f^+(u_i)+f^+(u_{i+1})}{2}\geq m-4\lambda f^+(m)$, which gives the lower bound. 
   For the case $\Delta^+f^+(\bar{u}_i)<0$, the proof is similar.
   \item If $d\hat{f}^{+(m)}_{i+\frac12}=\Delta^+f^+(\bar{u}_{i-1})$, the proof is the same as the previous case. 
\end{enumerate}

\end{proof}

\subsection{One-dimensional convection diffusion problems}
\label{sec-1dconffusion}
We consider the one-dimensional convection diffusion problems with periodic boundary conditions:
\setlength\arraycolsep{2pt}
$
u_t + f(u)_x = a(u)_{xx},\quad u(x,0) = u_0(x),
$
where $a'(u)\geq 0$. 
Let $\mathbf f^n$ denote the column vector with entries $f(u^n_1),\cdots, f(u^n_N)$. By notations introduced in Section \ref{sec-1dconvection}, the fourth-order compact finite difference with forward Euler  
can be denoted as:
\begin{equation}\label{1dconvectiondiffusionscheme}
 \mathbf{u}^{n+1} =\mathbf{u}^{n}-\frac{\Delta t}{\Delta x}W_1^{-1}D_x \mathbf{f}^{n}+\frac{\Delta t}{\Delta x^2} W_2^{-1}D_{xx}\mathbf{a}^{n}.
\end{equation}
Recall that we have abused the notation by using $W_1 f^n_i$ to denote the $i$-th entry of the vector $W_1 \mathbf f^n$ and we have defined $\bar u_i=W_1  u_i$.
We now define $$\tilde u_i=W_2  u_i.$$
Notice that $W_1$ and $W_2$  are both circulant thus they both can be diagonalized by the discrete Fourier matrix, so $W_1$ and $W_2$ commute. 
Thus we have
\begin{equation*}
 \tilde{\bar{u}}_i=(W_2W_1\mathbf{u})_i=(W_1W_2\mathbf{u})_i=\bar{\tilde{u}}_i.
\end{equation*}
Let $f^n_i=f(u^n_i)$ and $a^n_i=a(u^n_i)$, then the scheme \eqref{1dconvectiondiffusionscheme} can be written as
\[ \bar{\tilde{ u}}^{n+1}_i=\bar{\tilde{ u}}^n_i- \frac{\Delta t}{\Delta x} W_2D_x f^n_i+\frac{\Delta t}{\Delta x^2}W_1D_{xx}a^n_i.\]
\begin{thm}
\label{thm-weakmono-convecdiff}
Under the CFL constraint $\frac{\Delta t}{\Delta x} \max_{u}|f'(u)|\leq\frac16,\frac{\Delta t}{\Delta x^2}\max_{u}a'(u)\leq\frac{5}{24},$ if $u^n_i\in[m, M]$, then the scheme  \eqref{1dconvectiondiffusionscheme} satisfies that 
$ m\leq \bar{\tilde{ u}}^{n+1}_i\leq M.$
\end{thm}
\begin{proof}
Let $\lambda=\frac{\Delta t}{\Delta x}$  and $\mu=\frac{\Delta t}{\Delta x^2}$. We can  rewrite the scheme \eqref{1dconvectiondiffusionscheme} as
$$\mathbf{u}^{n+1}=\frac12(\mathbf{u}^{n}-2\lambda W_1^{-1}D_x \mathbf{f}^{n})+\frac12(\mathbf{u}^{n}+2\mu W_2^{-1}D_{xx}\mathbf{a}^{n}),$$
$$W_2W_1\mathbf{u}^{n+1}=\frac12W_2(W_1\mathbf{u}^{n}-2\lambda D_x \mathbf{f}^{n})+\frac12W_1(W_2\mathbf{u}^{n}+2\mu D_{xx}\mathbf{a}^{n}),$$
$$\bar{\tilde{ u}}^{n+1}_i=\frac12W_2(\bar{{ u}}^n_i-2\lambda D_x f^{n}_i)+\frac12W_1({\tilde{ u}}^n_i+2\mu D_{xx}a^{n}_i).$$
By Theorem \ref{thm-weakmono-convec}, we have $\bar{{ u}}^n_i-2\lambda D_x f^{n}_i\in [m, M]$.
We also have
\begin{align*}
&{\tilde{ u}}^n_i+2\mu D_{xx}{a}^{n}_i=\frac{1}{12}(u^n_{-1}+10u^n_i+u^n_{i+1})+2 \mu(a^n_{i-1}-2a^n_{i}+a^n_{i+1})\\
=& \left(\frac{5}{6}u^n_i-4\mu a^n_i\right)+
\left(\frac{1}{12}u^n_{i-1}+2\mu a^n_{i-1}\right)+\left(\frac{1}{12}u^n_{i+1}+2\mu a^n_{i+1}\right).
\end{align*}
Due to monotonicity under the CFL constraint and the assumption $a'(u)\geq 0$, we get ${\tilde{ u}}^n_i+2\mu D_{xx}{a}^{n}_i\in[m, M]$.
Thus we get $\bar{\tilde{ u}}^{n+1}_i\in[m, M]$ since it is a convex combination of $\bar{{ u}}^n_i-2\lambda D_x f^{n}_i$ and 
${\tilde{ u}}^n_i+2\mu D_{xx}a^{n}_i$.
\end{proof}

Given point values $u_i$ satisfying $\bar{\tilde{u}}_i\in[m,M]$ for any $i$, Lemma \ref{firstlemma} no longer holds
since $\bar{\tilde{u}}_i$ has a five-point  stencil.
However, the same three-point stencil limiter in Algorithm \ref{alg_limiter2} can still be used to enforce the lower and upper bounds.
Given  $\bar{\tilde{u}}_i=W_2 W_1 u_i$ $i=1,\cdots, N$, conceptually we can obtain the point values $u_i$ by first computing $\bar u_i=W_2^{-1}\bar{\tilde{u}}_i$ 
then computing $u_i=W_1^{-1}\bar{{u}}_i$. 
Thus we can apply the limiter in Algorithm \ref{alg_limiter2} twice to enforce $u_i\in[m, M]$:
\begin{enumerate}
 \item Given  $\bar{\tilde{u}}_i\in [m, M]$, compute $\bar u_i=W_2^{-1}\bar{\tilde{u}}_i$ which are not necessarily in the range $[m, M]$. 
 Then apply the limiter in Algorithm \ref{alg_limiter2} to $\bar u_i, i=1,\cdots, N$. Let $\bar v_i$ denote the output of the limiter. 
 Since we have 
 \[\bar{\tilde{u}}_i={\tilde{\bar u}}_i=\frac{1}{c+2}(\bar u_{i-1}+c \bar u_{i}+\bar u_{i+1}), \quad c=10,\]
 all discussions in Section \ref{sec-limiter} are still valid, thus we have $\bar v_i\in [m, M]$.
\item Compute $u_i=W_1^{-1}\bar{v}_i$. Apply the limiter in Algorithm \ref{alg_limiter2} to $u_i, i=1,\cdots, N$.
Let $v_i$ denote the output of the limiter. Then we have $v_i\in [m, M]$.
\end{enumerate}

\subsection{High order time discretizations}
\label{sec-highordertime}
For high order time discretizations, we can use strong stability preserving (SSP) Runge-Kutta and multistep methods, which are
convex combinations of formal forward Euler steps. Thus if using the limiter in Algorithm \ref{alg_limiter2} for fourth order compact finite difference schemes considered
in this section on each stage in a SSP Runge-Kutta method or each time step in a SSP multistep method, the bound-preserving property still holds. 

In the numerical tests, we will use a fourth order SSP multistep method and a fourth order SSP Runge-Kutta method \cite{gottlieb2011strong}.
Now consider solving $u_t = F (u)$. 
The SSP coefficient $C$ for a SSP time discretization is a constant so that the high order SSP time
discretization is stable in a norm or a semi-norm under the time step restriction $\Delta t\leq C \Delta t_0$, 
if under the time step restriction $\Delta t\leq \Delta t_0$ the forward Euler is stable in the same norm or semi-norm. 
The fourth order SSP Multistep method (with SSP coefficient $C_{ms}=0.1648$) and 
the fourth order SSP Runge-Kutta method  (with SSP coefficient $C_{rk}=1.508$)  will be used in the numerical tests. See \cite{gottlieb2011strong} for their definitions.


In Section \ref{sec-limiter} we have shown that  the limiters in Algorithm \ref{alg_limiter} and Algorithm \ref{alg_limiter2} are high order accurate
provided $u_i$ are high order accurate approximations to a smooth function $u(x)\in[m, M]$. This assumption holds for the numerical solution in a multistep method
in each time step, but it is no longer true for  inner stages in the Runge-Kutta method. So only SSP multistep methods
with  the limiter Algorithm \ref{alg_limiter2} are genuinely high order accurate schemes. For SSP Runge-Kutta methods,
using the bound-preserving limiter for compact finite difference schemes might result in an order reduction. The order reduction for bound-preserving limiters for finite volume
and DG schemes with Runge-Kutta methods was pointed out in \cite{zhang2010maximum} due to the same reason. However, such an order reduction in compact finite difference schemes 
is more prominent, as we will see in the numerical tests.

\section{Extensions to two-dimensional problems}
\label{sec-2d}
In this section we consider initial value problems on a square $[0,1]\times [0,1]$ with periodic boundary conditions.  
Let $(x_{i},y_j)=(\frac{i}{N_x},\frac{j}{N_y})$ $(i=1,\cdots, N_x,j=1, \cdots, N_y)$ be the uniform grid points on the domain $[0,1]\times[0,1]$. 
For a periodic function $f(x,y)$  on $[0,1]\times[0,1]$,
let $\mathbf{f}$ be a matrix  of size $N_x\times N_y$ with entries $f_{ij}$ representing point values $f(u_{ij})$.
We first define two linear operators $W_{1x}$ and $W_{1y}$  from $\mathbb R^{N_x\times N_y}$ to $\mathbb R^{N_x\times N_y}$:
\[W_{1x}\mathbf f=\frac16\begin{pmatrix}
              4 & 1 & & & 1\\
              1 & 4 & 1 & &  \\              
              &  \ddots & \ddots & \ddots & \\
              &  & 1 & 4 &1 \\
              1  & & & 1 & 4 \\
             \end{pmatrix}_{N_x\times N_x}
             \begin{pmatrix}
              f_{11} & f_{12} & \cdots & f_{1, N_y}\\
              f_{21} & f_{22} & \cdots & f_{2, N_y} \\              
              \vdots &  \vdots & \ddots  & \vdots\\
              f_{N_x-1,1}& f_{N_x-1,2} &\cdots  & f_{N_x-1,N_y} \\
              f_{N_x,1}  & f_{N_x,2}&\cdots & f_{N_x,N_y}  \\
             \end{pmatrix},\]
\[W_{1y}\mathbf f=
             \begin{pmatrix}
              f_{11} & f_{12} & \cdots & f_{1, N_y}\\
              f_{21} & f_{22} & \cdots & f_{2, N_y} \\              
              \vdots &  \vdots & \ddots  & \vdots\\
              f_{N_x-1,1}& f_{N_x-1,2} &\cdots  & f_{N_x-1,N_y} \\
              f_{N_x,1}  & f_{N_x,2}&\cdots & f_{N_x,N_y}  \\
             \end{pmatrix}\frac16\begin{pmatrix}
              4 & 1 & & & 1\\
              1 & 4 & 1 & &  \\              
              &  \ddots & \ddots & \ddots & \\
              &  & 1 & 4 &1 \\
              1  & & & 1 & 4 \\
             \end{pmatrix}_{N_y\times N_y}.\]             
We can define $W_{2x}$, $W_{2y}$, $D_x$, $D_y$, $W_{2x}$ and $W_{2y}$ similarly such that the subscript $x$ denotes the multiplication of the corresponding matrix from the left for 
the $x$-index and  the subscript $y$ denotes the multiplication of the corresponding matrix from the right for 
the $y$-index. We abuse the notations by using $W_{1x} f_{ij}$ to denote the $(i,j)$ entry of $W_{1x} \mathbf f$.  
We only discuss the forward Euler from now on since the discussion for high order SSP time discretizations are the same as in
Section \ref{sec-highordertime}.

\subsection{Two-dimensional convection equations}
\label{sec-2dconvec}
Consider solving the two-dimensional convection equation:
$u_t+f(u)_x+g(u)_y =0,\quad u(x,y,0)= u_0(x,y).$
By the our notations, the fourth order compact scheme with the forward Euler time discretization can be denoted as: 
\begin{equation}\label{2dscheme}
u^{n+1}_{ij}=u^n_{ij}-\frac{\Delta t}{\Delta x}W_{1x}^{-1}D_xf^n_{ij}-\frac{\Delta t}{\Delta y}W_{1y}^{-1}D_yg^n_{ij}.
\end{equation}
We define $\bar{\mathbf{u}}^{n}=W_{1x}W_{1y}\mathbf{u}^n$,
then by applying  $W_{1y}W_{1x}$ to both sides, \eqref{2dscheme} becomes 
 \begin{equation}\label{2dscheme2}
\bar{u}^{n+1}_{ij}=\bar{u}^n_{ij}-\frac{\Delta t}{\Delta x}W_{1y}D_xf^n_{ij}-\frac{\Delta t}{\Delta y}W_{1x}D_yg^n_{ij}.
\end{equation}
\begin{thm}
\label{thm-2dconvec}
 Under the CFL constraint 
 \begin{equation}\label{2dconveccfl}
  \frac{\Delta t}{\Delta x}\max_u|f'(u)|+\frac{\Delta t}{\Delta y}\max_u|g'(u)|\leq \frac13,
 \end{equation}
if ${u}^n_{ij}\in [m, M]$, then the scheme \eqref{2dscheme2} satisfies $\bar{u}^{n+1}_{ij}\in[m, M]$. 
\end{thm}
\begin{proof}
For convenience, we drop the time step $n$ in $u^n_{ij}$, $f^n_{ij}$, and introduce:
\[ U=\begin{pmatrix*}[l]
  u_{i-1, j+1} &u_{i, j+1} &u_{i+1, j+1} \\
  u_{i-1, j} &u_{i, j} &u_{i+1, j} \\
  u_{i-1, j-1} &u_{i, j-1} &u_{i+1, j-1}
 \end{pmatrix*}, \quad  F=\begin{pmatrix*}[l]
  f_{i-1, j+1} &f_{i, j+1} &f_{i+1, j+1} \\
  f_{i-1, j} &f_{i, j} &f_{i+1, j} \\
  f_{i-1, j-1} &f_{i, j-1} &f_{i+1, j-1}
 \end{pmatrix*}.\]
 Let $\lambda_1=\frac{\Delta t}{\Delta x}$ and  $\lambda_2=\frac{\Delta t}{\Delta y}$, then the scheme \eqref{2dscheme2} can be written as
\begin{eqnarray*}
 \bar{u}^{n+1}_{ij}& = &W_{1y}W_{1x} u^n_{ij}-\lambda_1W_{1y}D_xf^n_{ij}-\lambda_2W_{1x}D_yg^n_{ij},\\
&=&\frac{1}{36}\begin{pmatrix}
                       1 & 4 & 1\\
                       4 & 16 & 4\\
                       1 & 4 & 1 \end{pmatrix}
                       : U 
 - \frac{\lambda_1}{12}\begin{pmatrix}
                       -1 & 0 & 1\\
                       -4 & 0 & 4\\
                       -1 & 0 & 1 \end{pmatrix}
                       : F  - \frac{\lambda_2}{12}\begin{pmatrix}
                       1 & 4 & 1\\
                       0 & 0 & 0\\
                       -1 & -4 & -1 \end{pmatrix}
                       : G,                      
\end{eqnarray*}
where $:$ denotes the sum of all entrywise products in two matrices of the same size.
Obviously the right hand side above is a monotonically increasing function with respect to $u_{lm}$ for $i-1\leq l \leq i+1$, $j-1\leq m \leq j+1$
under the CFL constraint \eqref{2dconveccfl}. The monotonicity implies the bound-preserving result of $\bar{u}^{n+1}_{ij}$. 
\end{proof}

Given $\bar u_{ij}$, we can recover point values $u_{ij}$ by obtaining first $v_{ij}=W_{1x}^{-1}\bar u_{ij}$
then $u_{ij}=W_{1y}^{-1} v_{ij}$.
Thus similar to the discussions in Section \ref{sec-1dconffusion}, given point values $u_{ij}$ satisfying $\bar u_{ij}\in[m,M]$ for any $i$ and $j$, 
we can use the limiter in Algorithm \ref{alg_limiter2} in a dimension by dimension fashion to enforce $u_{ij}\in[m, M]$:
\begin{enumerate}
 \item Given  $\bar{u}_{ij}\in [m, M]$, compute $v_{ij}=W_{1x}^{-1}\bar{u}_{ij}$ which are not necessarily in the range $[m, M]$. 
 Then apply the limiter in Algorithm \ref{alg_limiter2} to $v_{ij}$ ($i=1,\cdots, N_x$) for each fixed $j$.
 Since we have 
 \[\bar{u}_{ij}=\frac{1}{c+2}(v_{i-1, j}+c v_{i,j}+v_{i+1,j}), \quad c=4,\]
 all discussions in Section \ref{sec-limiter} are still valid.  Let $\bar v_{ij}$ denote the output of the limiter,  thus we have $\bar v_{ij}\in [m, M]$.
\item Compute $u_{ij}=W_{1y}^{-1}\bar{v}_{ij}$. Then we have 
 \[\bar{v}_{ij}=\frac{1}{c+2}(u_{i, j-1}+c u_{i,j}+u_{i,j+1}), \quad c=4.\]
Apply the limiter in Algorithm \ref{alg_limiter2} to $u_{ij}$ ($j=1,\cdots, N_y$) for each fixed $i$.
Then the output values are in the range $[m, M]$.
\end{enumerate}

\subsection{Two-dimensional convection diffusion equations}
\label{sec-2dconfussion}
Consider the two-dimensional convection diffusion problem:
{\setlength\arraycolsep{2pt}
\[
u_t + f(u)_x + g(u)_y =a (u)_{xx}+b (u)_{xx},\quad
u(x,y,0)=u_0(x,y),
\]
where $a'(u)\geq 0$ and $b'(u)\geq 0$.
A fourth-order accurate compact finite difference scheme can be written as
\[ \frac{d\mathbf{u}}{dt} = -\frac{1}{\Delta x}W_{1x}^{-1}D_x \mathbf{f}-\frac{1}{\Delta y}W_{1y}^{-1}D_y \mathbf{g} + 
\frac{1}{\Delta x^2}W_{2x}^{-1}D_{xx}\mathbf{a}+\frac{1}{\Delta y^2}W_{2y}^{-1}D_{yy}\mathbf{b}.
\]
 Let $\lambda_1=\frac{\Delta t}{\Delta x}$,  $\lambda_2=\frac{\Delta t}{\Delta y}$,
 $\mu_1=\frac{\Delta t}{\Delta x^2}$ and  $\mu_2=\frac{\Delta t}{\Delta y^2}$. 
 With the forward Euler time discretization, the scheme becomes
\begin{equation}\label{2dcondiffscheme}
u^{n+1}_{ij}=u^n_{ij} -\lambda_1 W_{1x}^{-1}D_xf^n_{ij}-\lambda_2 W_{1y}^{-1}D_yg^n_{ij}
+\mu_1 W_{2x}^{-1}D_{xx}a^n_{ij}+ \mu_2 W_{2y}^{-1}D_{yy}b^n_{ij}.
\end{equation}

We first define
$\bar{\mathbf{u}}=W_{1x}W_{1y}{\mathbf u}$ and $\tilde{\mathbf{u}}=W_{2x}W_{2y}{\mathbf u}$, where $W_1=W_{1x}W_{1y}$ and $W_2=W_{2x}W_{2y}$. 
Due to the fact $W_1W_2=W_2W_1$, we have
\begin{equation*}
 \tilde{\bar{\mathbf{u}}}=W_{2x}W_{2y}(W_{1x}W_{1y}{\mathbf u})=W_{1x}W_{1y}(W_{2x}W_{2y}{\mathbf u})=\bar{\tilde{\mathbf{u}}}.
\end{equation*}
The scheme \eqref{2dcondiffscheme} is equivalent to the following form:
\begin{eqnarray*}\label{2dcondiffschemeinm}
\tilde{\bar{u}}^{n+1}_{ij}&=\tilde{\bar{u}}^n_{ij} -\lambda_1 W_{1y}W_{2x}W_{2y}D_xf^n_{ij}-\lambda_2 W_{1x}W_{2x}W_{2y}D_yg^n_{ij}\\
&\quad +\mu_1 W_{1x}W_{1y}W_{2y}D_{xx}a^n_{ij}+ \mu_2 W_{1x}W_{1y}W_{2x}D_{yy}b^n_{ij}.
\end{eqnarray*}

\begin{thm}
 Under the CFL constraint 
 \begin{equation}\label{2dconvecdiffcfl}
  \frac{\Delta t}{\Delta x}\max_u|f'(u)|+\frac{\Delta t}{\Delta y}\max_u|g'(u)|\leq \frac16,
 \frac{\Delta t}{\Delta x^2}\max_u a'(u)+\frac{\Delta t}{\Delta y^2}\max_u b'(u)\leq \frac{5}{24},
 \end{equation}
 if ${u}^n_{ij}\in [m, M]$, then the scheme \eqref{2dcondiffscheme} satisfies $\tilde{\bar{u}}^{n+1}_{ij}\in[m, M]$. 
\end{thm}
\begin{proof}
By using $\tilde{\bar{u}}^{n}_{ij}=\frac12 \tilde{\bar{u}}^{n}_{ij}+ \frac12 \bar{ \tilde u}^{n}_{ij}$, we obtain
\begin{eqnarray*}
\tilde{\bar{u}}^{n+1}_{ij}&=&\frac12W_{2x}W_{2y}[\bar{u}^n_{ij} -2 \lambda_1W_{1y}D_xf^n_{ij}-2\lambda_2 W_{1x}D_yg^n_{ij}]\\
&& +\frac12 W_{1x}W_{1y}[\tilde{u}^n_{ij}+2\mu_1 W_{2y}D_{xx}a^n_{ij}+ 2 \mu_2 W_{2x}D_{yy}b^n_{ij}].
\end{eqnarray*}
Let 
$\bar{v}_{ij}=\bar{u}^n_{ij} -2 \lambda_1W_{1y}D_xf^n_{ij}-2\lambda_2 W_{1x}D_yg^n_{ij},
 \tilde{w}_{ij}=\tilde{u}^n_{ij}+2\mu_1 W_{2y}D_{xx}a^n_{ij}+ 2 \mu_2 W_{2x}D_{yy}b^n_{ij}.
$
Then by the same discussion as in the proof of Theorem \ref{thm-2dconvec}, we can show $\bar{v}_{ij}\in [m, M]$.
For  $\tilde{w}_{ij}$, it can be written as 
\[\tilde{w}_{ij}= \frac{1}{144}\begin{pmatrix}
                       1 & 10 & 1\\
                       10 & 100 & 10\\
                       1 & 10 & 1 \end{pmatrix}: U+\frac{\mu_1}{6}\begin{pmatrix}
                       1 & -2 & 1\\
                       10 & -20 & 10\\
                       1 & -2 & 1 \end{pmatrix}: A+ \frac{\mu_2}{6} \begin{pmatrix}
                       1 & 10 & 1\\
                       -2 & -20 & -2\\
                       1 & 10 & 1 \end{pmatrix}: B , \]
\[ A=\begin{pmatrix*}[l]
  a_{i-1, j+1} &a_{i, j+1} &a_{i+1, j+1} \\
  a_{i-1, j} &a_{i, j} &a_{i+1, j} \\
  a_{i-1, j-1} &a_{i, j-1} &a_{i+1, j-1}
 \end{pmatrix*}, \quad  B=\begin{pmatrix*}[l]
  b_{i-1, j+1} &b_{i, j+1} &b_{i+1, j+1} \\
  b_{i-1, j} &b_{i, j} &b_{i+1, j} \\
  b_{i-1, j-1} &b_{i, j-1} &b_{i+1, j-1}
 \end{pmatrix*}.\]
 
Under the CFL constraint \eqref{2dconvecdiffcfl}, $\tilde{w}_{ij}$ is a monotonically increasing function of $u^n_{ij}$ involved thus $\tilde{w}_{ij}\in[m, M]$.
Therefore, $\tilde{\bar{u}}^{n+1}_{ij}\in[m, M]$.
\end{proof}

Given $\tilde{\bar u}_{ij}$, we can recover point values $u_{ij}$ by obtaining first $\tilde u_{ij}=W_{1x}^{-1}W_{1y}^{-1}\tilde{\bar u}_{ij}$
then $u_{ij}=W_{2x}^{-1}W_{2y}^{-1} \tilde u_{ij}$.
Thus similar to the discussions in the previous subsection, given point values $u_{ij}$ satisfying $\tilde{\bar u}_{ij} \in[m,M]$ for any $i$ and $j$, 
we can use the limiter in Algorithm \ref{alg_limiter2} dimension by dimension  several times to enforce $u_{ij}\in[m, M]$:
\begin{enumerate}
 \item Given  $\tilde{\bar u}_{ij}\in [m, M]$, compute $\tilde u_{ij}=W_{1x}^{-1}W_{1y}^{-1}\tilde{\bar u}_{ij}$ and
 apply the limiting algorithm in the previous subsection to ensure  $\tilde u_{ij}\in [m, M]$.
 \item  Compute $v_{ij}=W_{2x}^{-1}\tilde{u}_{ij}$ which are not necessarily in the range $[m, M]$. 
 Then apply the limiter in Algorithm \ref{alg_limiter2} to $v_{ij}$ for each fixed $j$.
 Since we have 
 \[\tilde{u}_{ij}=\frac{1}{c+2}(v_{i-1, j}+c v_{i,j}+v_{i+1,j}),  c=10,\]
 all discussions in Section \ref{sec-limiter} are still valid.  Let $\tilde v_{ij}$ denote the output of the limiter,  thus we have $\tilde v_{ij}\in [m, M]$.
\item Compute $u_{ij}=W_{2y}^{-1}\tilde {v}_{ij}$. Then we have 
 $\tilde {v}_{ij}=\frac{1}{c+2}(u_{i, j-1}+c u_{i,j}+u_{i,j+1}), \quad c=10.$
Apply the limiter in Algorithm \ref{alg_limiter2} to $u_{ij}$ for each fixed $i$.
Then the output values are in the range $[m, M]$.
\end{enumerate}

\section{Higher order extensions}
\label{sec-highorder}
The weak monotonicity  may not hold for 
a generic compact finite difference operator. See \cite{lele1992compact} for a general discussion of compact finite difference schemes.
In this section we demonstrate how to construct a higher order accurate compact finite difference scheme satisfying the weak monotonicity.
 Following Section \ref{sec-1d} and Section \ref{sec-2d},
we can use these compact finite difference operators to construct higher order accurate bound-preserving schemes. 
\subsection{Higher order compact finite difference operators}
Consider a compact finite difference approximation to the first order derivative in the following form:
\begin{equation} \label{highorder1}
\beta_1 f'_{i-2}+\alpha_1 f'_{i-1}+f_i'+ \alpha_1 f'_{i+1}+ \beta_1 f'_{i+2}=b_1\frac{f_{i+2}-f_{i-2}}{4\Delta x}+a_1\frac{f_{i+1}-f_{i-1}}{2\Delta x},
\end{equation}
where $\alpha_1, \beta_1, a_1, b_1$ are constants to be determined.
To obtain a sixth order accurate approximation, there are many choices for $\alpha_1, \beta_1, a_1, b_1$. 
To ensure the approximation in \eqref{highorder1} satisfies the weak monotonicity for solving scalar conservation laws
under some CFL condition, we need $\alpha_1>0, \beta_1>0$.
By requirements above, we obtain 
\begin{equation}\label{highorder2}
 \beta_1=\frac{1}{12}(-1+3\alpha_1),\quad a_1=\frac{2}{9}(8-3\alpha_1),\quad b_1=\frac{1}{18}(-17+57\alpha_1),\quad \alpha_1>\frac13.
\end{equation}
With \eqref{highorder2}, the approximation \eqref{highorder1} is sixth order accurate and satisfies the weak monotonicity as discussed in Section \ref{sec-1dconvection}.
The truncation error of the approximation \eqref{highorder1} and \eqref{highorder2} is $\frac{4}{7!}(9\alpha_1-4)\Delta x^6f^{(7)}+\mathcal O(\Delta x^8)$, so if setting
\begin{equation}\label{highorder3}
 \alpha_1=\frac{4}{9},\quad \beta_1=\frac{1}{36},\quad a_1=\frac{40}{27},\quad b_1=\frac{25}{54},
\end{equation}
 we have an eighth order accurate approximation satisfying the weak monotonicity.

Now consider the fourth order compact finite difference approximations to the second derivative in the following form:
\[ \beta_2 f''_{i-2}+\alpha_2 f''_{i-1}+f_i'+ \alpha_2 f''_{i+1}+ \beta_2 f''_{i+2}=b_2\frac{f_{i+2}-2f_i+f_{i-2}}{4\Delta x^2}+a_2\frac{f_{i+1}-2f_i+f_{i-1}}{\Delta x^2},\]
\[a_2=\frac{1}{3}(4-4\alpha_2-40\beta_2),\quad b_2=\frac{1}{3}(-1+10\alpha_2+46\beta_2).\]
with the truncation error $\frac{-4}{6!}(-2+11\alpha_2-124\beta_2)\Delta x^4f^{(6)}$. 
The fourth order scheme discussed in Section \ref{sec-1d} is the special case with
$\alpha_2=\frac{1}{10},\quad \beta_2=0,\quad a_2=\frac{6}{5},\quad b_2=0.$
If $\beta_2=\frac{11\alpha_2-2}{124}$, we get a family of sixth-order schemes satisfying the weak monotonicity:
\begin{equation}
 a_2=\frac{-78\alpha_2+48}{31},\quad b_2=\frac{291\alpha_2-36}{62},\quad \alpha_2 >0.
\end{equation}
The truncation error of the sixth order approximation is $\frac{4}{31\cdot 8!}(1179\alpha_2-344)\Delta x^6f^{(8)}$.
Thus we obtain an eighth order approximation satisfying the weak monotonicity if
\begin{equation}\label{8thorderdiff}
\alpha_2=\frac{344}{1179}, \beta_2=\frac{23}{2358}, a_2=\frac{320}{393}, b_2=\frac{310}{393}, 
\end{equation}
with truncation error $\frac{-172}{5676885}\Delta x^8f^{(10)}$.

\subsection{Convection problems}
\label{sec-highorderconvec}
For the rest of this section, we will mostly focus on the family of sixth order schemes since the eighth order accurate scheme is a special case of this family. 
For  $u_t+f(u)_x=0$ with periodic boundary conditions on the interval $[0,1]$,
 we get the following semi-discrete scheme:
\[
\frac{d}{dt}\mathbf{u}=-\frac{1}{\Delta x}\widetilde{W}_1^{-1}\widetilde{D}_x\mathbf{f},
\]
    {\setlength\arraycolsep{2pt}{\[\widetilde{W}_1\mathbf{u}=\frac{\beta_1}{1+2\alpha_1+2\beta_1}\begin{pmatrix}
              \frac{1}{\beta_1} & \frac{\alpha_1}{\beta_1} & 1 & & & 1& \frac{\alpha_1}{\beta_1}\\
              \frac{\alpha_1}{\beta_1} & \frac{1}{\beta_1} & \frac{\alpha_1}{\beta_1} & 1 & & & 1 \\              
              1 & \frac{\alpha_1}{\beta_1} & \frac{1}{\beta_1} & \frac{\alpha_1}{\beta_1} & 1 & & \\
              &  \ddots & \ddots & \ddots & \ddots & \ddots &\\
              & & 1 & \frac{\alpha_1}{\beta_1} & \frac{1}{\beta_1} & \frac{\alpha_1}{\beta_1} & 1 \\
              1 & & & 1 & \frac{\alpha_1}{\beta_1} & \frac{1}{\beta_1} & \frac{\alpha_1}{\beta_1}\\
              \frac{\alpha_1}{\beta_1} & 1 & & & 1 & \frac{\alpha_1}{\beta_1} & \frac{1}{\beta_1}\\
             \end{pmatrix}
             \begin{pmatrix}
 u_1 \\
 u_2\\
 u_3\\
 \vdots\\
 u_{N-2}\\
 u_{N-1}\\
 u_N
 \end{pmatrix},\]}
   {\setlength\arraycolsep{2pt}{\[\widetilde{D}_x\mathbf{f} =
 \frac{1}{4(1+2\alpha_1+2\beta_1)}\begin{pmatrix}
              0 & 2a_1 & b_1 & & & -b_1 & -2a_1\\
              -2a_1 & 0 & 2a_1 & b_1 & & & -b_1 \\              
              -b_1 & -2a_1 & 0 & 2a_1 & b_1 & \\
              & \ddots & \ddots & \ddots & \ddots & \ddots & \\
              & & -b_1 & -2a_1 & 0 & 2a_1 & b_1 \\
              b_1 &  & & -b_1 & -2a_1 & 0 & 2a_1 \\
              2a_1 & b_1 & & & -b_1 & -2a_1 & 0\\
             \end{pmatrix} 
  \begin{pmatrix}
 f_1 \\
 f_2\\
 f_3\\
 \vdots\\
 f_{N-2}\\
 f_{N-1}\\
 f_N
 \end{pmatrix},\]}
where $f_i$ and $u_i$ are point values of functions $f(u(x))$ and $u(x)$ at uniform grid points $x_i$ $(i=1,\cdots,N)$ respectively.
We have a family of sixth-order compact schemes with forward Euler time discretization:
\begin{equation}\label{6thordercon}
  \mathbf{u}^{n+1}=\mathbf{u}^n-\frac{\Delta t}{\Delta x}\widetilde{W}_1^{-1}\widetilde{D}_x\mathbf{f}.
\end{equation}
Define $\bar{\mathbf u}=\widetilde{W}_1 \mathbf u$ and $\lambda=\frac{\Delta t}{\Delta x}$,
then scheme \eqref{6thordercon} can be written as 
$$\bar{u}^{n+1}_i=\bar{u}^n_i-\frac{\lambda}{4(1+2\alpha_1+2\beta_1)}(b_1f^n_{i+2}+2a_1f^n_{i+1}-2a_1f^n_{i-1}-b_1f^n_{i-2}).$$
Following the lines in Section \ref{sec-1dconvection}, we can easily conclude that the scheme \eqref{6thordercon} satisfies
$\bar{u}^{n+1}_i\in[m,M]$ if $u_i^n\in[m,M]$,
under the CFL constraint \[\frac{\Delta t}{\Delta x}|f'(u)|\leq \min\{\frac{9}{8-3\alpha_1},\frac{6(3\alpha_1-1)}{57\alpha_1-17}\}.\]

Given $\bar u_i \in[m, M]$, we also need a limiter to enforce $u_i\in[m, M]$. 
Notice that $\bar{u}_i$ has a five-point stencil instead of a three-point stencil in Section \ref{sec-limiter}. Thus in general the extensions
of Section \ref{sec-limiter} for sixth order schemes are more complicated. However, we can still use the same limiter as in 
Section \ref{sec-limiter} because the five-diagonal matrix $\widetilde{W}_1$ can be represented as  a product of two tridiagonal matrices. 

Plugging in $\beta_1=\frac{1}{12}(-1+3\alpha_1)$, we have
$
 \widetilde{W}_1=\widetilde{W}^{(1)}_1\widetilde{W}^{(2)}_1,
 $
where
 \[\widetilde{W}^{(1)}_1=\frac{1}{c^{(1)}_1+2}\begin{pmatrix}
              c^{(1)}_1 & 1 & & &  1\\
              1 & c^{(1)}_1 & 1 &  & \\              
               & \ddots & \ddots & \ddots & \\
               & & 1 & c^{(1)}_1 &1 \\
              1  & & & 1 & c^{(1)}_1 \\
             \end{pmatrix},
c^{(1)}_1=\frac{6\alpha_1}{3\alpha_1-1}-\frac{\sqrt{2}\sqrt{7-24\alpha_1+27\alpha_1^2}}{\sqrt{1-6\alpha_1+9\alpha_1^2}},
             \]
 
 \[
 \widetilde{W}^{(2)}_1=\frac{1}{c^{(2)}_1+2}\begin{pmatrix}
              c^{(2)}_1 & 1 & & &  1\\
              1 & c^{(2)}_1 & 1 &  & \\              
               & \ddots & \ddots & \ddots & \\
               & & 1 & c^{(2)}_1 &1 \\
              1  & & & 1 & c^{(2)}_1 \\
             \end{pmatrix},
c^{(2)}_1=\frac{6\alpha_1}{3\alpha_1-1}+\frac{\sqrt{2}\sqrt{7-24\alpha_1+27\alpha_1^2}}{\sqrt{1-6\alpha_1+9\alpha_1^2}}.
\]

In other words, $\bar{\mathbf{u}}= \widetilde W_1\mathbf{u}= \widetilde W_1^{(1)}  \widetilde W_1^{(2)}\mathbf{u}$.
Thus following the limiting procedure in Section \ref{sec-1dconffusion},
we can still use the same limiter in Section \ref{sec-limiter} twice to enforce the bounds of point values if $c^{(1)}_1,c^{(2)}_1\geq 2$, 
which implies $\frac13<\alpha_1\leq \frac{5}{9}$. 
In this case we have $\min\{\frac{9}{8-3\alpha_1},\frac{6(3\alpha_1-1)}{57\alpha_1-17}\}=\frac{6(3\alpha_1-1)}{57\alpha_1-17}$, thus the CFL for the weak monotonicity
becomes
$\lambda|f'(u)|\leq \frac{6(3\alpha_1-1)}{57\alpha_1-17}$.
We summarize the results in the following theorem.
\begin{thm}
Consider a family of sixth order accurate schemes \eqref{6thordercon} with 
  $$\beta_1=\frac{1}{12}(-1+3\alpha_1),\quad a_1=\frac{2}{9}(8-3\alpha_1),\quad b_1=\frac{1}{18}(-17+57\alpha_1),\quad \frac13<\alpha_1\leq \frac{5}{9},$$
 which includes the eighth order scheme \eqref{highorder3} as a special case. If $u^n_i\in [m, M]$ for all $i$, under the CFL constraint
 $\frac{\Delta t}{\Delta x}\max_u|f'(u)|\leq \frac{6(3\alpha_1-1)}{57\alpha_1-17},$
 we have $\bar u^{n+1}_i\in[m, M]$.
\end{thm}

Given point values $u_i$ satisfying $\widetilde W_1^{(1)}  \widetilde W_1^{(2)} u_i= \widetilde W_1 u_i=\bar{u}_i\in[m,M]$ for any $i$,
we can apply the limiter in Algorithm \ref{alg_limiter2} twice to enforce $u_i\in[m, M]$:
\begin{enumerate}
 \item Given  $\bar{{u}}_i\in [m, M]$, compute $v_i=[\widetilde W_1^{(1)}]^{-1}\bar{u}_i$ which are not necessarily in the range $[m, M]$. 
 Then apply the limiter in Algorithm \ref{alg_limiter2} to $v_i, i=1,\cdots, N$. Let $\bar v_i$ denote the output of the limiter. 
 Since we have 
 $\bar{u}_i=\frac{1}{c_1^{(1)}+2}(v_{i-1}+c_1^{(1)} v_{i}+v_{i+1}), c_1^{(1)}>2,$
 all discussions in Section \ref{sec-limiter} are still valid, thus we have $\bar v_i\in [m, M]$.
\item Compute $u_i=[\widetilde W_1^{(2)}]^{-1}\bar{v}_i$. Apply the limiter in Algorithm \ref{alg_limiter2} to $u_i, i=1,\cdots, N$.
Since we have 
 $\bar{v}_i=\frac{1}{c_1^{(2)}+2}(u_{i-1}+c_1^{(2)} u_{i}+u_{i+1}), c_1^{(2)}>2,$
 all discussions in Section \ref{sec-limiter} are still valid, thus the output are in  $[m, M]$.
\end{enumerate}

\subsection{Diffusion problems}
For simplicity we only consider the diffusion problems and the extension to convection diffusion problems
can be easily discussed following Section \ref{sec-1dconffusion}.
For the one-dimensional scalar  diffusion equation $u_t=g(u)_{xx}$ with $g'(u)\geq 0$ and  periodic boundary conditions on an interval  $[0,1]$,
we get the  sixth order semi-discrete scheme:
$
 \frac{d}{dt}\mathbf{u}=\frac{1}{\Delta x^2}\widetilde{W}^{-1}_2\widetilde{D}_{xx}\mathbf{g},
$
where
\[
 \widetilde{W}_2\mathbf{u}=\frac{\beta_2}{1+2\alpha_2+2\beta_2}\begin{pmatrix}
              \frac{1}{\beta_2} & \frac{\alpha_2}{\beta_2} & 1 & & & 1& \frac{\alpha_2}{\beta_2}\\
              \frac{\alpha_2}{\beta_2} & \frac{1}{\beta_2} & \frac{\alpha_2}{\beta_2} & 1 & & & 1 \\              
              1 & \frac{\alpha_2}{\beta_2} & \frac{1}{\beta_2} & \frac{\alpha_2}{\beta_2} & 1 & & \\
              &  \ddots & \ddots & \ddots & \ddots & \ddots &\\
              & & 1 & \frac{\alpha_2}{\beta_2} & \frac{1}{\beta_2} & \frac{\alpha_2}{\beta_2} & 1 \\
              1 & & & 1 & \frac{\alpha_2}{\beta_2} & \frac{1}{\beta_2} & \frac{\alpha_2}{\beta_2}\\
              \frac{\alpha_2}{\beta_2} & 1 & & & 1 & \frac{\alpha_2}{\beta_2} & \frac{1}{\beta_2}\\
             \end{pmatrix}
             \begin{pmatrix}
 u_1 \\
 u_2\\
 u_3\\
 \vdots\\
 u_{N-2}\\
 u_{N-1}\\
 u_N
 \end{pmatrix},\]
 
{\setlength\arraycolsep{0pt}{

 \mbox{\footnotesize
$\widetilde{D}_{xx}\mathbf{g} =
 \frac{1}{4(1+2\alpha_2+2\beta_2)}\begin{pmatrix}
              -8a_2-2b_2 & 4a_2 & 2b_2 & & & 2b_2 & 4a_2\\
              4a_2 & -8a_2-2b_2 & 4a_2 & 2b_2 & & & 2b_2 \\              
              2b_2 & 4a_2 & -8a_2-2b_2 & 4a_2 & 2b_2 & \\
              & \ddots & \ddots & \ddots & \ddots & \ddots & \\
              & & 2b_2 & 4a_2 & -8a_2-2b_2 & 4a_2 & 2b_2 \\
              2b_2 &  & & 2b_2 & 4a_2 & -8a_2-2b_2 & 4a_2 \\
              4a_2 & 2b_2 & & & 2b_2 & 4a_2 & -8a_2-2b_2\\
             \end{pmatrix} 
  \begin{pmatrix}
 g_1 \\
 g_2\\
 g_3\\
 \vdots\\
 g_{N-2}\\
 g_{N-1}\\
 g_N
 \end{pmatrix},$}
 }  
 where $g_i$ and $u_i$ are values of functions $g(u(x))$ and $u(x)$ at$x_i$ respectively.
 
As in the previous subsection, we prefer to factor $\widetilde{W}_2$ as a product of two tridiagonal matrices. 
Plugging in $\beta_2=\frac{11\alpha_2-2}{124}$, we have:
$
 \widetilde{W}_2=\widetilde{W}^{(1)}_2\widetilde{W}^{(2)}_2,
 $ where
{\setlength\arraycolsep{2pt}{ 
  \begin{eqnarray*}\widetilde{W}^{(1)}_2=\frac{1}{c^{(1)}_2+2}\begin{pmatrix}
              c^{(1)}_2 & 1 & & &  1\\
              1 & c^{(1)}_2 & 1 & &  \\                    
               & \ddots & \ddots & \ddots & \\
               & & 1 & c^{(1)}_2 &1 \\
              1  & & & 1 & c^{(1)}_2 \\
             \end{pmatrix},
c^{(1)}_2=\frac{62\alpha_2}{11\alpha_2-2}-\frac{\sqrt{2}\sqrt{128-726\alpha_2+2043\alpha_2^2}}{\sqrt{4-44\alpha_2+121\alpha_2^2}},\nonumber\\
 \widetilde{W}^{(2)}_2=\frac{1}{c^{(2)}_2+2}\begin{pmatrix}
              c^{(2)}_2 & 1 & & & 1\\
              1 & c^{(2)}_2 & 1 & &  \\ 
               & \ddots & \ddots & \ddots & \\
               & & 1 & c^{(2)}_2 &1 \\
              1  & & & 1 & c^{(2)}_2 \\
             \end{pmatrix},
c^{(2)}_2=\frac{62\alpha_2}{11\alpha_2-2}+\frac{\sqrt{2}\sqrt{128-726\alpha_2+2043\alpha_2^2}}{\sqrt{4-44\alpha_2+121\alpha_2^2}}.\nonumber
\end{eqnarray*}
}
To have $c^{(1)}_2,c^{(2)}_2 \geq 2$, we need $\frac{2}{11}<\alpha_2 \leq \frac{60}{113}$.
The forward Euler gives
\begin{equation}\label{6thordercondiff}
 \mathbf{u}^{n+1}=\mathbf{u}^n+\frac{\Delta t}{\Delta x^2}\widetilde{W}_2^{-1}\widetilde{D}_{xx}\mathbf{g}.
\end{equation}
Define $\tilde{u}_i=\widetilde{W}_2 u_i$ and $\mu=\frac{\Delta t}{\Delta x^2}$, then the scheme \eqref{6thordercondiff} can be written as 
\[
 \tilde u^{n+1}_i= \tilde u^{n}_i+ \frac{\mu}{4(1+2\alpha_2+2\beta_2)}\left[2b_2g^n_{i-2}+4a_2g^n_{i-1}+(-8a_2-2b_2)g^n_i+4a_2g^n_{i+1}+2b_2g^n_{i+2}\right].
\]
\begin{thm}
Consider a family of sixth order accurate schemes \eqref{6thordercondiff} with 
  $$\beta_2=\frac{11\alpha_2-2}{124},  a_2=\frac{-78\alpha_2+48}{31},\quad b_2=\frac{291\alpha_2-36}{62},\quad  \frac{2}{11}<\alpha_2 \leq \frac{60}{113},$$
 which includes the eighth order scheme \eqref{8thorderdiff} as a special case. If $u^n_i\in [m, M]$ for all $i$, under the CFL 
 $\frac{\Delta t}{\Delta x^2}g'(u)<\frac{124}{3(116-111\alpha_2)},$
 the scheme satisfies $\tilde u^{n+1}\in[m, M]$.
\end{thm}

As in the previous subsection, given point values $u_i$ satisfying $\widetilde W_2^{(1)}  \widetilde W_2^{(2)} u_i= \widetilde W_2 u_i=\tilde{u}_i\in[m,M]$ for any $i$,
we can apply the limiter in Algorithm \ref{alg_limiter2} twice to enforce $u_i\in[m, M]$.
The matrices $\widetilde{W}_1$  and $\widetilde{W}_2$ commute because they are both circulant matrices thus diagonalizable by the discrete Fourier matrix. 
The discussion for the sixth order scheme solving convection diffusion problems is also straightforward.

\section{Extensions to general boundary conditions}
\label{sec-bc}
Since the compact finite difference operator is implicitly defined thus any extension to 
other type boundary conditions is not straightforward. In order to maintain the weak monotonicity, the boundary conditions must be properly treated. 
In this section we demonstrate  a high order accurate boundary treatment preserving the weak monotonicity for inflow and outflow boundary conditions.
For convection problems, we can easily construct a fourth order accurate boundary scheme. For convection diffusion problems, it is much more complicated
to achieve weak monotonicity near the boundary thus a straightforward discussion gives us a third order accurate boundary scheme.
\subsection{Inflow-outflow boundary conditions for convection problems}
For simplicity, we consider the following initial boundary value problem on the interval $[0,1]$ as an example:
$
 u_t+f(u)_x=0, \quad u(x,0)=u_0(x),\quad u(0,t)=L(t),
$
where we assume $f'(u)>0$ so that the inflow boundary condition at the left cell end is a well-posed boundary condition. 
The boundary condition at $x=1$ is not specified thus understood as an outflow boundary condition. 
We further assume $u_0(x)\in[m, M]$ and $L(t)\in[m, M]$ so that the exact solution is in $[m, M]$.

Consider a uniform grid with $x_i=i\Delta x$ for $i=0, 1, \cdots, N, N+1$ and $\Delta x=\frac{1}{N+1}$. 
Then a fourth order semi-discrete compact finite difference scheme is given by
{\setlength\arraycolsep{2pt}\begin{eqnarray*}
\frac{d}{dt}\frac16\begin{pmatrix}
              1 & 4 & 1 & & \\            
               &  \ddots & \ddots & \ddots\\
                & & 1 & 4 & 1
             \end{pmatrix}
             \begin{pmatrix}
 u_{0} \\
 \vdots\\
 u_{N+1}
 \end{pmatrix}=
 \frac{1}{2\Delta x}\begin{pmatrix}
              -1 &0 & 1 & & & & \\
                & \ddots & \ddots & \ddots  \\
                & & -1 & 0 & 1 \\
             \end{pmatrix}
             \begin{pmatrix}
 f_0 \\
 \vdots\\
 f_{N+1}
 \end{pmatrix}.
\end{eqnarray*}}
With forward Euler time discretization, the scheme is equivalent to 
\begin{equation}\label{inflowoutflow}
 \bar{u}_i^{n+1}=\bar{u}_i^n-\frac{1}{2}\lambda(f_{i+1}^n-f_{i-1}^n),\quad i=1, \cdots, N.
\end{equation}
Here $u_0^n=L(t^n)$ is given as boundary condition for any $n$. Given $u_i^n$ for $i=0,1,\cdots, N+1$,
the scheme \eqref{inflowoutflow} gives $\bar u^{n+1}_i$ for $i=1, \cdots, N$, from which we still need $u_{N+1}^{n+1}$ to recover interior point values $u^{n+1}_i$ for $i=1, \cdots, N$.  

Since the boundary condition at $x_{N+1}=1$ can be implemented as outflow, we can use $\bar u^{n+1}_i$ for $i=1, \cdots, N$ to obtain a reconstructed  $u_{N+1}^{n+1}$. If there is a cubic polynomial $p_i(x)$ so that $u_{i-1}, u_i, u_{i+1}$ are its point values at $x_{i-1}, x_i, x_{i+1}$, then
$\frac{1}{2\Delta x}\int_{x_{i-1}}^{x_{i+1}} p_i(x)\,dx=\frac16 u_{i-1}+\frac46 u_i+\frac16 u_{i+1}=\bar u_i,$
due to the exactness of the Simpson's quadrature rule for cubic polynomials. 
To this end, we can consider a unique cubic polynomial $p(x)$ satisfying four equations:
$\frac{1}{2\Delta x}\int_{x_{j-1}}^{x_{j+1}} p(x)\,dx=\bar u_j^{n+1},\quad j=N-3, N-2, N-1, N.$
If $\bar u_j^{n+1}$ are fourth order accurate approximations to $\frac16 u(x_{j-1}, t^{n+1})+\frac46 u(x_{j}, t^{n+1})+\frac16 u(x_{j+1}, t^{n+1})$,  
then $p(x)$ is a fourth order accurate approximation to $u(x,t^{n+1})$ on the interval $[x_{N-4}, x_{N+1}]$.
So we get a fourth order accurate $u_{N+1}^{n+1}$ by  
\begin{equation}\label{interpolation}
 p(x_{N+1})=-\frac{2}{3}\bar{u}_{N-3}+\frac{17}{6}\bar{u}_{N-2}-\frac{14}{3}\bar{u}_{N-1}+\frac{7}{2}\bar{u}_{N}.
\end{equation}
Since \eqref{interpolation} is not a convex linear combination, $p(x_{N+1})$ may not lie in the bound $[m,M]$. Thus to ensure $u^{n+1}_{N+1}\in[m, M]$ we can define
\begin{equation}\label{filter}
 u^{n+1}_{N+1}:=\max\{\min\{p(x_{N+1}),M\},m\}.
\end{equation}

Obviously Theorem \ref{thm-weakmono-convec} still holds for the scheme \eqref{inflowoutflow}. 
For the forward Euler time discretization, we can implement the bound-preserving scheme as follows:
\begin{enumerate}
 \item Given $u_i^n$ for all $i$, compute $\bar u^{n+1}_i$ for $i=1, \cdots, N$ by   \eqref{inflowoutflow}.
 \item Obtain boundary values $u^{n+1}_0=L(t^{n+1})$ and $u^{n+1}_{N+1}$ by \eqref{interpolation} and \eqref{filter}. 
 \item Given $\bar u^{n+1}_i$ for $i=1, \cdots, N$ and two boundary values  $u^{n+1}_0$ and $u^{n+1}_{N+1}$, 
 recover point values $u^{n+1}_i$ for $i=1, \cdots, N$ by solving the tridiagonal linear system (the superscript $n+1$ is omitted):
 \[ \frac16\begin{pmatrix}
              4 & 1 & & &\\              
              1 & 4 & 1 & &\\
               & \ddots & \ddots & \ddots&\\
              &  & 1 & 4 & 1\\
              & && 1 & 4
             \end{pmatrix}
             \begin{pmatrix}
 u_{1} \\
 u_{2}\\
 \vdots\\
  u_{N-1}\\
  u_{N}
 \end{pmatrix}
 =\begin{pmatrix}
 \bar u_{1}-\frac16 u_0 \\
 \bar u_{2}\\
 \vdots\\
  \bar u_{N-1}\\
  \bar u_{N}-\frac16 u_{N+1}
 \end{pmatrix}.\]
 \item Apply the limiter in Algorithm \ref{alg_limiter2} to the point values  $u^{n+1}_i$ for $i=1, \cdots, N$. 
\end{enumerate}

\subsection{Dirichlet boundary conditions for one-dimensional convection diffusion equations}
\label{sec-Dirichlet}
Consider the initial boundary value problem for a one-dimensional scalar convection diffusion equation on the interval $[0,1]$:
\begin{equation}\label{convectiondiffusionDirichlet}
 u_t+f(u)_x=g(u)_{xx},\quad u(x,t)=u_0(x),\quad u(0,t)=L(t),\quad u(1,t)=R(t),
\end{equation}
where $g'(u)\geq 0$.
We further assume $u_0(x)\in[m, M]$ and $L(t), R(t)\in[m, M]$ so that the exact solution is in $[m, M]$.

We demonstrate how to treat the boundary approximations so that the scheme still satisfies some weak monotonicity such that 
a certain convex combination of point values is in the range $[m, M]$ at the next time step. 
Consider a uniform grid with $x_i=i\Delta x$ for $i=0, 1, \cdots, N, N+1$ where $\Delta x=\frac{1}{N+1}$. 
The fourth order compact finite difference approximations at the interior points can be written as:
\begin{eqnarray*}
W_1
 \begin{pmatrix}
 f_{x,1} \\
 f_{x,2}\\
 \vdots\\
 f_{x,N-1}\\
 f_{x,N}
 \end{pmatrix}  =
\frac{1}{\Delta x}D_x
  \begin{pmatrix}
 f_1 \\
 f_2\\
 \vdots\\
 f_{N-1}\\
 f_N
 \end{pmatrix}
 +
 \begin{pmatrix}
 -\frac{f_{x,0}}{6}-\frac{f_0}{2\Delta x} \\
 0\\
 \vdots\\
 0\\
 -\frac{f_{x,N+1}}{6}+\frac{f_{N+1}}{2\Delta x}
 \end{pmatrix}
 ,\end{eqnarray*}
 {\setlength\arraycolsep{2pt}
 \begin{eqnarray*}
W_1=\frac16\begin{pmatrix}
              4 & 1 & & &  \\
              1 & 4 & 1 & \\   
              & \ddots & \ddots & \ddots & \\
               & & 1 & 4 & 1\\
               & & & 1 & 4 
             \end{pmatrix}, \quad 
             D_x= \frac12 \begin{pmatrix}
              0 & 1 & & &  \\
              -1 & 0 & 1 &   \\ 
               & \ddots & \ddots & \ddots & \\
               & & -1 & 0 & 1 \\
               & & & -1 & 0 \\
             \end{pmatrix},
 \end{eqnarray*}
}
\begin{eqnarray*}
W_2
 \begin{pmatrix}
 g_{xx,1} \\
 g_{xx,2}\\
 \vdots\\
 g_{xx,N-1}\\
 g_{xx,N}
 \end{pmatrix}  =
\frac{1}{\Delta x^2} D_{xx}
  \begin{pmatrix}
 g_1 \\
 g_2\\
 \vdots\\
 g_{N-1}\\
 g_N
 \end{pmatrix}
 +
 \begin{pmatrix}
 -\frac{g_{xx,0}}{12}+\frac{g_0}{\Delta x^2} \\
 0\\
 \vdots\\
 0\\
 -\frac{g_{xx,N+1}}{12}+\frac{g_{N+1}}{\Delta x^2}
 \end{pmatrix}
 ,\end{eqnarray*}
 
{\setlength\arraycolsep{2pt}\begin{eqnarray*}
 W_{2}=\frac{1}{12} \begin{pmatrix}
              10& 1 & & &  \\
              1 & 10 & 1 & &  \\    
               & \ddots & \ddots & \ddots & \\
               & & 1 & 10 &1 \\
               & & & 1 & 10 
             \end{pmatrix},\quad
D_{xx}=\begin{pmatrix}
              -2 & 1 & & & \\
              1 & -2 & 1 & &  \\   
               & \ddots & \ddots & \ddots & \\
               & & 1 & -2 &1 \\
               & & & 1 & -2 
             \end{pmatrix},
\end{eqnarray*}
}
where $f_{x,i}$ and $g_{xx,i}$ denotes the values of $f(u)_x$ and $g(u)_{xx}$ at $x_i$ respectively. Let 
\begin{eqnarray*}
F= \begin{pmatrix}
 -\frac{f_{x,0}}{6}-\frac{f_0}{2\Delta x} \\
 0\\
 \vdots\\
 0\\
 -\frac{f_{x,N+1}}{6}+\frac{f_{N+1}}{2\Delta x}
 \end{pmatrix},\quad
G= \begin{pmatrix}
 -\frac{g_{xx,0}}{12}+\frac{g_0}{\Delta x^2} \\
 0\\
 \vdots\\
 0\\
 -\frac{g_{xx,N+1}}{12}+\frac{g_{N+1}}{\Delta x^2}
 \end{pmatrix}.
\end{eqnarray*}
Define $W:=W_1W_2=W_2W_1$. Here $W_2$ and $W_1$ commute because they have the same eigenvectors, which is due to the fact that $2W_2-W_1$ is the identity matrix.
Let $\mathbf u=\begin{pmatrix}
                u_1 & u_2 & \cdots & u_N
               \end{pmatrix}^T$, 
$\mathbf{f}=\begin{pmatrix}
 f(u_1) &
 f(u_2) &
 \cdots &
 f(u_N)
 \end{pmatrix}^T$ and $\mathbf{g}=\begin{pmatrix}
 g(u_1) &
 g(u_2) &
 \cdots &
 g(u_N)
 \end{pmatrix}^T$.                    
Then a fourth order compact finite difference approximation to \eqref{convectiondiffusionDirichlet} at the interior grid points is
$\frac{d}{dt}\mathbf u+ W_{1}^{-1}(\frac{1}{\Delta x}D_x \mathbf{f}+F)=W_{2}^{-1}(\frac{1}{\Delta x^2}D_{xx} \mathbf{g}+G)$
which is equivalent to 
\begin{eqnarray*}
 \frac{d}{dt}(W\mathbf u)+\frac{1}{\Delta x} W_2D_x\mathbf{f}-\frac{1}{\Delta x^2} W_1D_{xx}\mathbf{g}=-W_2F+W_1G.
\end{eqnarray*}
If $u_i(t)=u(x_i,t)$ where $u(x,t)$ is the exact solution to the problem, then it satisfies
\begin{equation}
u_{t,i}+f_{x,i}=g_{xx, i},
\label{exactgridvalue}
\end{equation}
where $u_{t,i}=\frac{d}{dt} u_{i}(t)$, $f_{x,i}=f(u_i)_{x}$ and $g_{xx,i}=g(u_i)_{xx}$.
If we use \eqref{exactgridvalue} to simplify $-W_2F+W_1G$, then the scheme is still fourth order accurate. 
In other words, setting $-f_{x,i}+g_{xx, i}=u_{t,i}$ does not affect the accuracy. Plugging \eqref{exactgridvalue} in the original $-W_2F+W_1G$, 
we can redefine $-W_2F+W_1G$ as
\begin{eqnarray*}
 -W_2F+W_1G: =\begin{pmatrix}
 -\frac{1}{18}u_{t,0}+\frac{1}{12}f_{x,0}+\frac{5}{12\Delta x}f_0+\frac{2}{3\Delta x^2}g_0 \\
 -\frac{1}{72}u_{t,0}+\frac{1}{24}f_0+\frac{1}{6\Delta x^2}g_0\\
 0\\
 \vdots\\
 0\\
 -\frac{1}{72}u_{t,N+1}-\frac{1}{24}f_{N+1}+\frac{1}{6\Delta x^2}g_{N+1}\\
 -\frac{1}{18}u_{t,N+1}+\frac{1}{12}f_{x,N+1}-\frac{5}{12\Delta x}f_{N+1}+\frac{2}{3\Delta x^2}g_{N+1}
 \end{pmatrix}.
\end{eqnarray*}
So we now consider the following fourth order accurate scheme:
\begin{equation}
 \label{c-dDirichletscheme}
\frac{d}{dt}(W\mathbf u)+\frac{1}{\Delta x}W_2D_x\mathbf f-\frac{1}{\Delta x^2}W_1D_{xx}\mathbf g =\begin{pmatrix}
 -\frac{1}{18}u_{t,0}+\frac{1}{12}f_{x,0}+\frac{5}{12\Delta x}f_0+\frac{2}{3\Delta x^2}g_0 \\
 -\frac{1}{72}u_{t,0}+\frac{1}{24}f_0+\frac{1}{6\Delta x^2}g_0\\
 0\\
 \vdots\\
 0\\
 -\frac{1}{72}u_{t,N+1}-\frac{1}{24}f_{N+1}+\frac{1}{6\Delta x^2}g_{N+1}\\
 -\frac{1}{18}u_{t,N+1}+\frac{1}{12}f_{x,N+1}-\frac{5}{12\Delta x}f_{N+1}+\frac{2}{3\Delta x^2}g_{N+1}
 \end{pmatrix}.
\end{equation}
The first equation in \eqref{c-dDirichletscheme} is 
\begin{eqnarray*}
 \frac{d}{dt}(\frac{4u_0+41u_1+14u_2+u_3}{72})=\frac{1}{24\Delta x}(10f_0+f_1-10f_2-f_3)
 +\frac{1}{6\Delta x^2}(4g_0-7g_1+2g_2+g_3)+\frac{1}{12}f_{x,0}.
\end{eqnarray*}
After multiplying $\frac{72}{60}=\frac{6}{5}$ to both sides, it becomes
\begin{eqnarray}
\label{c-dDirichletschemefirstline}
 \frac{d}{dt}(\frac{4u_0+41u_1+14u_2+u_3}{60})=\frac{1}{20\Delta x}(10f_0+f_1-10f_2-f_3)\nonumber\\
 +\frac{1}{5\Delta x^2}(4g_0-7g_1+2g_2+g_3)+\frac{1}{10}f_{x,0}.
\end{eqnarray}

In order for the scheme \eqref{c-dDirichletschemefirstline} to satisfy a weak monotonicity in the sense that $\frac{4u^{n+1}_0+41u^{n+1}_1+14u^{n+1}_2+u^{n+1}_3}{60}$ 
in \eqref{c-dDirichletschemefirstline} with forward Euler can be written as 
a monotonically increasing function of $u^n_i$ under some CFL constraint, we still need to find an approximation to $f(u)_{x,0}$ using only $u_0, u_1, u_2, u_3$,
with which we have a straightforward third order approximation to $f(u)_{x,0}$:
\begin{equation}\label{bc-third}
f_{x,0}=\frac{1}{\Delta x}(-\frac{11}{6}f_0+3f_1-\frac32f_2+\frac13f_3)+\mathcal O(\Delta x^3).                                                   
\end{equation}
Then \eqref{c-dDirichletschemefirstline} becomes
\begin{eqnarray}
\label{c-dDirichletschemefirstline2}
  \frac{d}{dt}(\frac{4u_0+41u_1+14u_2+u_3}{60})=\frac{1}{60\Delta x}(19f_0+21f_1-39f_2-f_3)
  \nonumber\\
 +\frac{1}{5\Delta x^2}(4g_0-7g_1+2g_2+g_3).
\end{eqnarray}

The second to second last equations of \eqref{c-dDirichletscheme} can be written as
\begin{eqnarray}\label{c-dDirichletscheme_middle}
 \frac{d}{dt}(\frac{u_{i-2}+14u_{i-1}+42u_i+14u_{i+1}+u_{i+2}}{72})=\frac{1}{24\Delta x}(f_{i-2}+10f_{i-1}\\
 -10f_{i+1}-f_{i+2})+\frac{1}{6\Delta x^2}(g_{i-2}+2g_{i-1}-6g_i+2g_{i+1}+g_{i+2}),\quad 2\leq i \leq N-1,\nonumber
\end{eqnarray}
which satisfies a straightforward weak monotonicity under some CFL constraint.

The last equation in \eqref{c-dDirichletscheme} is 
\begin{eqnarray*}
  \frac{d}{dt}(\frac{4u_{N+1}+41 u_{N}+14u_{N-1}+u_{N-2}}{72})=\frac{1}{24\Delta x}(f_{N-2}+10f_{N-1}-f_N\nonumber\\
  -10f_{N+1})+\frac{1}{6\Delta x^2}(g_{N-2}+2g_{N-1}-7g_{N}+4g_{N+1})+\frac{1}{12}f_{x,N+1}.
\end{eqnarray*}
After multiplying $\frac{72}{60}=\frac{6}{5}$ to both sides, it becomes 
\begin{eqnarray*}
  \frac{d}{dt}(\frac{u_{N-2}+14u_{N-1}+41 u_{N}+4u_{N+1}}{60})=\frac{1}{20\Delta x}(f_{N-2}+10f_{N-1}-f_N\nonumber\\
  -10f_{N+1})+\frac{1}{5\Delta x^2}(g_{N-2}+2g_{N-1}-7g_{N}+4g_{N+1})+\frac{1}{10}f_{x,N+1}.
\end{eqnarray*}
Similar to the boundary scheme at $x_0$, 
we should use a third-order approximation:
\begin{equation}
f_{x,N+1}=\frac{1}{\Delta x}(-\frac13f_{N-2}+\frac32f_{N-1}-3f_N+\frac{11}{6}f_{N+1})+\mathcal O(\Delta x^3).                                                   
\end{equation}
Then the boundary scheme at $x_{N+1}$ becomes
{\setlength\arraycolsep{2pt}\begin{eqnarray}
\label{c-dDirichletschemelastline}
   \frac{d}{dt}(\frac{u_{N-2}+14u_{N-1}+41 u_{N}+4u_{N+1}}{60})=\frac{1}{60\Delta x}(f_{N-2}+39f_{N-1}-21f_N\nonumber\\
  -19f_{N+1})+\frac{1}{5\Delta x^2}(g_{N-2}+2g_{N-1}-7g_{N}+4g_{N+1}).\quad
\end{eqnarray}
}
 
 To summarize the full semi-discrete scheme, we can represent the third order scheme
 \eqref{c-dDirichletschemefirstline2}, \eqref{c-dDirichletscheme_middle} and \eqref{c-dDirichletschemelastline},
for the Dirichlet boundary conditions as:
\begin{eqnarray*}
\frac{d}{dt}\widetilde{W}\tilde{\mathbf u}=-\frac{1}{\Delta x} \widetilde{D}_xf(\tilde{\mathbf u})+\frac{1}{\Delta x^2}\widetilde{D}_{xx}g(\tilde{\mathbf u}),
\end{eqnarray*}
where 
\begin{eqnarray*}
\label{1d-w-tilde}
 \widetilde{W}= \frac{1}{72}\begin{pmatrix}
              \frac{24}{5} & \frac{246}{5} & \frac{84}{5} & \frac65 & & & & \\
              1 & 14 & 42 & 14 & 1 & &  \\     
               & \ddots & \ddots & \ddots & \ddots & \ddots\\
               & & 1 & 14 & 42 & 14 & 1 \\
               & & & \frac65 & \frac{84}{5} & \frac{246}{5} & \frac{24}{5} 
             \end{pmatrix}_{N\times (N+2)},
             \tilde{\mathbf u}=\begin{pmatrix}
 u_{0} \\
 u_{1}\\
 \vdots\\
 u_{N}\\
 u_{N+1}
             \end{pmatrix}_{(N+2)\times 1},\end{eqnarray*}
\[   \widetilde{D}_x=\frac{1}{24}\begin{pmatrix}
                    -\frac{38}{5} & -\frac{42}{5} & \frac{78}{5} & \frac25 & & & &\\
                    -1 & -10 &  0 & 10 & 1 & &\\
                     & \ddots & \ddots & \ddots & \ddots & \ddots \\
                     & & -1 & -10 &  0   & 10 &1 \\
                     & & & -\frac25 & -\frac{78}{5} & \frac{42}{5} & \frac{38}{5} 
                   \end{pmatrix}_{N\times (N+2)},
\widetilde{D}_{xx}=\frac{1}{6}
\begin{pmatrix}
 \frac{24}{5} & -\frac{42}{5} & \frac{12}{5} & \frac65 & & & &\\
1 & 2 & -6 & 2 & 1 & &\\
 & \ddots & \ddots & \ddots & \ddots & \ddots \\
 & & 1 & 2 & -6 & 2 & 1\\
 & & & \frac65 & \frac{12}{5} & -\frac{42}{5} & \frac{24}{5}
\end{pmatrix}_{N\times (N+2)}. \nonumber 
\]
Let $\bar{\mathbf u}=\widetilde{W}\tilde{\mathbf u}$, $\lambda=\frac{\Delta t}{\Delta x}$ and $\mu=\frac{\Delta t}{\Delta x^2}$. With forward Euler, it becomes
\begin{eqnarray}
\label{bcscheme}
\bar{u}^{n+1}_i=\bar{u}^{n}_i-\frac12\lambda\widetilde{D}_x\tilde{f}_i+\mu\widetilde{D}_{xx}\tilde{g}_i, \quad i=1,\cdots, N.
\end{eqnarray}
We state the weak monotonicity without proof.
\begin{thm}
Under the CFL constraint 
$\frac{\Delta t}{\Delta x} \max_u|f'(u)|\leq \frac{4}{19},\frac{\Delta t}{\Delta x^2} \max_u g'(u)\leq \frac{695}{1596},$
if $u^n_i\in[m, M]$, then the scheme  \eqref{bcscheme} satisfies $\bar u^{n+1}_i\in [m, M]$. 
\end{thm}
We notice that 
\begin{eqnarray*}
 \bar{u}_1^{n+1} & = & \frac{1}{60}(4u_0^{n+1}+41u_1^{n+1}+14u_2^{n+1}+u_3^{n+1})=\frac{u_0^{n+1}+4u_1^{n+1}+u_2^{n+1}}{6} + \frac{1}{10}\frac{u_1^{n+1}+4u_2^{n+1}+u_3^{n+1}}{6}-\frac{1}{10}u_0^{n+1},
\end{eqnarray*}
\begin{eqnarray*}
 \bar{u}_N^{n+1} & = & \frac{1}{60}(u_{N-2}^{n+1}+14u_{N-1}^{n+1}+41u_{N}^{n+1}+4u_{N+1}^{n+1})=  \frac{1}{10} \frac{u_{N-2}^{n+1}+4u_{N-1}^{n+1}+u_{N}^{n+1}}{6} + \frac{u_{N-1}^{n+1}+4u_{N}^{n+1}+u_{N+1}^{n+1}}{6}-\frac{1}{10}u_{N+1}^{n+1}.
\end{eqnarray*}
Recall that the boundary values are given: $u^{n+1}_0=L(t^{n+1})\in[m, M]$ and $u^{n+1}_{N+1}=R(t^{n+1})\in[m, M]$, so we have
\begin{eqnarray*}
 \frac{10}{11}\frac{u_0^{n+1}+4u_1^{n+1}+u_2^{n+1}}{6} + \frac{1}{11}\frac{u_1^{n+1}+4u_2^{n+1}+u_3^{n+1}}{6}\leq \frac{10}{11}M+\frac{1}{11}M=M,\\
 \frac{10}{11}\frac{u_0^{n+1}+4u_1^{n+1}+u_2^{n+1}}{6} + \frac{1}{11}\frac{u_1^{n+1}+4u_2^{n+1}+u_3^{n+1}}{6}\geq \frac{10}{11}m+\frac{1}{11}m=m,\\
\frac{1}{11} \frac{u_{N-2}^{n+1}+4u_{N-1}^{n+1}+u_{N}^{n+1}}{6} + \frac{10}{11}\frac{u_{N-1}^{n+1}+4u_{N}^{n+1}+u_{N+1}^{n+1}}{6}\leq \frac{10}{11}M+\frac{1}{11}M=M,\\
\frac{1}{11} \frac{u_{N-2}^{n+1}+4u_{N-1}^{n+1}+u_{N}^{n+1}}{6} + \frac{10}{11}\frac{u_{N-1}^{n+1}+4u_{N}^{n+1}+u_{N+1}^{n+1}}{6}\geq \frac{10}{11}m+\frac{1}{11}m=m.
\end{eqnarray*}
Thus define $\mathbf{w}^{n+1}=\begin{pmatrix}
 w_{1}^{n+1},
 w_{2}^{n+1},
 w_{3}^{n+1},
 \dots,
 w_{N-1}^{n+1},
 w_{N}^{n+1}
\end{pmatrix}^T$ as follows and we  have:
\begin{eqnarray*}
m \leq w_i^{n+1}:&=&\bar{u}_i^{n+1}\leq M,\quad i=2,\cdots,N-1,\\
m\leq w_1^{n+1}:&=&\frac{10}{11}\frac{u_0^{n+1}+4u_1^{n+1}+u_2^{n+1}}{6} + \frac{1}{11}\frac{u_1^{n+1}+4u_2^{n+1}+u_3^{n+1}}{6}\leq M,\\
m\leq w_N^{n+1}:&=&\frac{1}{11}\frac{u_{N-3}^{n+1}+4u_{N-2}^{n+1}+u_{N-1}^{n+1}}{6}+\frac{10}{11}\frac{u_{N-2}^{n+1}+4u_{N-1}^{n+1}+u_{N}^{n+1}}{6}\leq M.
\end{eqnarray*}
By the notations above, we get
\begin{eqnarray}\label{reconstructweighting}
\mathbf{w}^{n+1}=K\bar{\mathbf u}^{n+1}+\mathbf{u}^{n+1}_{bc}=\widetilde{\widetilde{W}}\tilde{\mathbf u},
\end{eqnarray}
\begin{eqnarray*}
             K=\begin{pmatrix}
              \frac{10}{11} & & &  \\
              & 1 & & & & \\    
               & & \ddots & & \\
               & & & 1 &  \\
               & & & & & \frac{10}{11}
             \end{pmatrix}_{N\times N},
\mathbf{u}_{bc}=\frac{1}{11}\begin{pmatrix}
 u_{0} \\
 0\\
 \vdots\\
 0\\
 u_{N+1}\end{pmatrix}_{N\times 1},
 \widetilde{\widetilde{W}}=\frac{1}{72}\begin{pmatrix}
              \frac{120}{11} & \frac{492}{11} & \frac{168}{11} & \frac{12}{11} & & & & \\
              1 & 14 & 42 & 14 & 1 & &  \\  
               & \ddots & \ddots & \ddots & \ddots & \ddots\\
               & & 1 & 14 & 42 & 14 & 1 \\
               & & & \frac{12}{11} & \frac{168}{11} & \frac{492}{11} & \frac{120}{11} 
             \end{pmatrix}_{N\times(N+2)}.
\end{eqnarray*}

We notice that $\widetilde{\widetilde{W}}$ can be factored as a product of two tridiagonal matrices:
{\setlength\arraycolsep{2pt}
\begin{equation*}
\frac{1}{72}\begin{pmatrix}
              \frac{120}{11} & \frac{492}{11} & \frac{168}{11} & \frac{12}{11} & & & & \\
              1 & 14 & 42 & 14 & 1 & &  \\  
               & \ddots & \ddots & \ddots & \ddots & \ddots\\
               & & 1 & 14 & 42 & 14 & 1 \\
               & & & \frac{12}{11} & \frac{168}{11} & \frac{492}{11} & \frac{120}{11} 
             \end{pmatrix}=
  \frac{1}{12}\begin{pmatrix}
              \frac{120}{11}&  \frac{12}{11} & & & & \\
              1 & 10 & 1 & & & \\ 
               & \ddots & \ddots & \ddots & \\
               & & 1 & 10 &1 \\
               & & &  \frac{12}{11} &  \frac{120}{11}
             \end{pmatrix}_{N\times N}\frac{1}{6}\begin{pmatrix}
              1 & 4 & 1 & & & & & \\
              & 1 & 4 & 1 & & & & \\  
               & & \ddots & \ddots & \ddots &\\
               & & & 1 & 4 & 1 & \\
               & & & & 1 & 4 & 1 
             \end{pmatrix}_{N\times(N+2)},
\end{equation*}}
which can be denoted as $\widetilde{\widetilde{W}} = \widetilde{W}_2\widetilde{W}_1.$
Fortunately, all the diagonal entries of $\widetilde{W}_1$ and  $\widetilde{W}_2$ are in the form of $\frac{c}{c+2}, c>2$. 
So given $\bar u_i=\widetilde W u_i\in [m, M]$, we construct $w^{n+1}_i\in[m,M]$.
We can apply the limiter in Algorithm \ref{alg_limiter2} twice to enforce $u_i\in [m, M]$:
\begin{enumerate}
 \item Given $u^n_i$ for all $i$, use the scheme \eqref{bcscheme} to obtain $\bar u^{n+1}_i\in [m, M]$ for $i=1, \cdots, N$. 
 Then construct $w^{n+1}_i\in [m, M]$ for $i=1, \cdots, N$ by  \eqref{reconstructweighting}.
 \item Notice that $\widetilde{W}_2$ is a matrix of size $N\times N$. Compute $\mathbf v=\widetilde{W}_2^{-1}\mathbf{w}^{n+1}$. 
 Apply the limiter in Algorithm \ref{alg_limiter2} to $v_i$ and let $\bar v_i$ denote the output values. 
 Since we have $\widetilde{W}_2 v_i\in [m, M]$, i.e.,
 \begin{eqnarray*}
m\leq& \frac{10}{11}v_1+\frac{1}{11}v_2&\leq M,\\  
m\leq& \frac{1}{12}v_1+\frac{10}{12}v_2+\frac{1}{12}v_3&\leq M,\\  
&\vdots&\\
m\leq& \frac{1}{12}v_{N-2}+\frac{10}{12}v_{N-1}+\frac{1}{12}v_{N}&\leq M,\\  
m\leq& \frac{1}{11}v_{N-1}+\frac{10}{11}v_N&\leq M.
 \end{eqnarray*}
Following the discussions in Section \ref{sec-limiter}, it implies $\bar v_i\in[m, M]$.
\item Obtain values of $u^{n+1}_i$, $i=1,\cdots, N$ by solving a $N\times N$ system:
\[\frac{1}{6}\begin{pmatrix}
               4 & 1 & & & &  \\
               1 & 4 & 1 & & & \\   
                & \ddots & \ddots & \ddots &\\
                & & 1 & 4 & 1 & \\
                & & & 1 & 4 
             \end{pmatrix}
             \begin{pmatrix}
               u^{n+1}_1\\ u^{n+1}_2\\ \vdots \\ u^{n+1}_{N-1}\\ u^{n+1}_N
             \end{pmatrix}
             =\begin{pmatrix}
               \bar v_1\\ \bar v_2\\ \vdots \\ \bar v_{N-1}\\ \bar v_N
             \end{pmatrix}-\frac{1}{6}\mathbf{u}^{n+1}_{bc}.
             \]
\item Apply the limiter in Algorithm \ref{alg_limiter2} to $u^{n+1}_i$  to ensure $u^{n+1}_i\in[m, M]$. 
\end{enumerate}

\section{Numerical tests}
\label{sec-test}

\subsection{One-dimensional problems with periodic boundary conditions}
In this subsection, we test the fourth order and eighth order accurate compact finite difference schemes with the bound-preserving limiter.
The time step is taken to satisfy both the CFL condition required for weak monotonicity in Theorem \ref{thm-weakmono-convec} and Theorem \ref{thm-weakmono-convecdiff}
and the SSP coefficient for high order SSP time discretizations.

\begin{ex}{One-dimensional linear convection equation}.\label{1dconvectioneqnexmp}
Consider $u_t+u_x=0$ with and initial condition $u_0(x)$ and periodic boundary conditions on the interval $[0,2\pi]$. 
The $L^1$ and $L^{\infty}$ errors for the fourth order scheme with a smooth initial condition at time $T=10$ are listed in Table \ref{table-accuracy1} where $\Delta x=\frac{2\pi}{N}$, the time step is taken as $\Delta t=C_{ms} \frac13 \Delta x$ for the multistep method,
and $\Delta t=5 C_{ms} \frac13 \Delta x$ for the  Runge-Kutta method so that the number of spatial discretization operators computed  is the same
as in the one for the multistep method. 
We can observe the fourth order accuracy for the multistep method and obvious order reductions for the Runge-Kutta method.

The errors for smooth initial conditions at time $T=10$ for the eighth order accurate scheme are listed in Table \ref{1dconvection8th}.
For the eighth order accurate scheme, the time step to achieve the weak monotonicity is $\Delta t=C_{ms} \frac{6}{25} \Delta x$ 
for the  fourth-order SSP multistep method. On the other hand, we need to set $\Delta t=\Delta x^2$ in fourth order accurate time discretizations
to verify the eighth order spatial accuracy. 
To this end,
the time step is taken as $\Delta t=C_{ms} \frac{6}{25} \Delta x^2$ for the  multistep method,
and $\Delta t=5 C_{ms} \frac{6}{25} \Delta x^2$ for the Runge-Kutta method.
We can observe the eighth order accuracy for the multistep method and the order reduction for $N=160$ is due to the roundoff errors. 
We can also see an obvious order reduction for the Runge-Kutta method.

\begin{table}[ht]
\centering
\caption{The fourth order accurate compact finite difference scheme with the bound-preserving limiter on a uniform $N$-point grid for the linear convection with initial data $u_0(x)=\frac12+\sin^4(x)$.}
\resizebox{\textwidth}{!}{  \begin{tabular}{|c|c c|c c|cc| c c|}
\cline{1-9}
& \multicolumn{4}{|c|} {Fourth order SSP multistep} & \multicolumn{4}{|c|} {Fourth order SSP Runge-Kutta}  \\
\cline{1-9}
\hline  N &  $L^1$ error  &  order & $L^\infty$ error & order & $L^1$ error & order & $L^\infty$ error & order\\
\cline{1-9}
\hline
\cline{1-9}
 $20$  & 3.44E-2 & - & 6.49E-2 & -& 3.41E-2 & - & 6.26E-2  & - \\
\cline{1-9}
 $40$  & 3.12E-3 & 3.47 & 6.19E-3 & 3.39 & 3.14E-3 & 3.44 & 6.62E-3 & 3.24\\
\cline{1-9}
$80$  & 1.82E-4 &  4.10 & 2.95E-4 & 4.39 & 1.86E-4 & 4.08 & 3.82E-4 & 4.11  \\ 
\cline{1-9}
$160$  & 1.10E-5 & 4.05 & 1.85E-5 & 4.00 & 1.29E-5 & 3.85 & 4.48E-5 & 3.09 \\
\cline{1-9}
$320$  & 6.81E-7 & 4.02 & 1.15E-6 & 4.01 & 1.42E-6 & 3.18 & 1.03E-5 & 2.13\\
\hline
\cline{1-9}
\hline
\end{tabular}}
\label{table-accuracy1}
\end{table}

\begin{table}[ht]
\centering
\caption{The eighth order accurate compact finite difference scheme with the bound-preserving limiter on a uniform $N$-point grid for 
the linear convection with initial data $u_0(x)=\frac12+\frac12\sin^4(x)$.}
\begin{tabular}{|c|c c|c c|cc| c c|}
\cline{1-9}
& \multicolumn{4}{|c|} {Fourth order SSP multistep} & \multicolumn{4}{|c|} {Fourth order SSP Runge-Kutta}  \\
\cline{1-9}
\hline  N &  $L^1$ error  &  order & $L^\infty$ error & order & $L^1$ error & order & $L^\infty$ error & order\\
\hline
 $10$  & 6.31E-2 & - & 1.01E-1 & -& 6.44E-2 & - & 9.58E-2 & - \\
\cline{1-9}
 $20$  & 3.35E-5 & 7.55 & 5.59E-4 & 7.49 & 3.39E-4 & 7.57 & 5.79E-4 & 7.37 \\
\cline{1-9}
 $40$  & 9.58E-7 & 8.45 & 1.49E-6 & 8.55 & 1.52E-6 & 7.80 & 4.32E-6 & 7.06 \\ 
 \cline{1-9}
 $80$  & 3.50E-9 & 8.10 & 5.51E-9 & 8.08 & 5.34E-8 & 4.83 & 2.31E-7 & 4.23  \\
 \cline{1-9}
 $160$  & 6.57E-11 & 5.74 & 1.01E-10 & 5.77 & 2.40E-9 & 4.48 & 1.45E-8 & 3.99 \\
\hline
\end{tabular}
\label{1dconvection8th}
\end{table}

Next, we consider the following discontinuous initial data:
\begin{equation}
\label{initial-step}
 u_0(x)=\left\{\begin{array}{ll}
 1, & \quad \textrm{if }\quad 0<x\leq\pi,\\
 0, & \quad \textrm{if }\quad \pi<x\leq 2\pi.
\end{array}\right.
\end{equation}
See Figure \ref{1dconvection_shockimg} for the performance of the bound-preserving limiter and the TVB limiter on the fourth order scheme.
We observe that the TVB limiter can reduce oscillations but cannot remove the overshoot/undershoot. When both limiters
are used, we can obtain a non-oscillatory bound-preserving numerical solution.
See Figure \ref{1dconvection_shock8thimg} for the performance of the bound-preserving limiter on the eighth order scheme. 

 \begin{figure}[ht]
 \subfigure[without any limiter]{\includegraphics[scale=0.45]{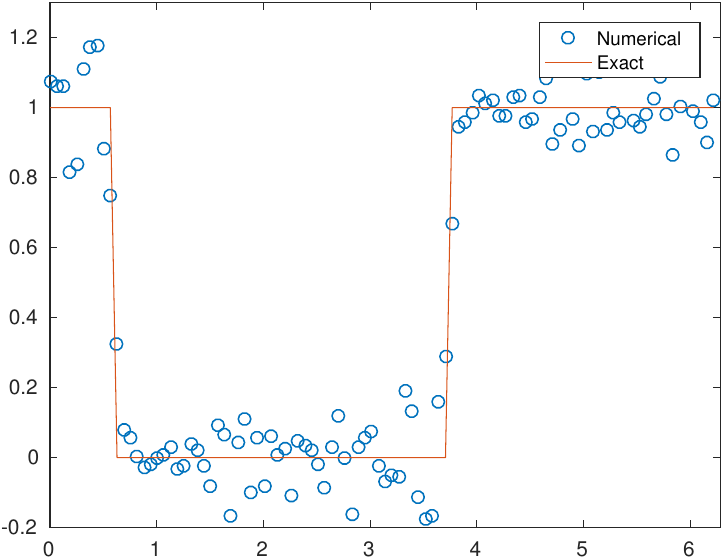} }
 \hspace{.1in}
 \subfigure[with only the bound-preserving limiter]{\includegraphics[scale=0.45]{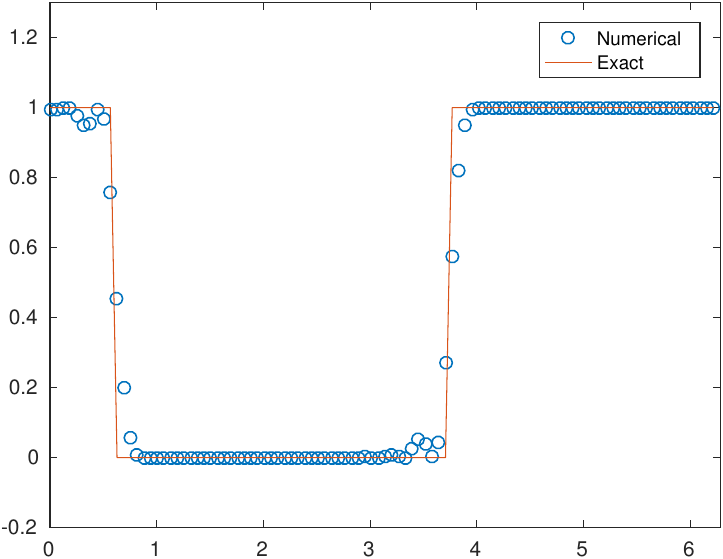}}\\
  \subfigure[with only the TVB limiter]{\includegraphics[scale=0.45]{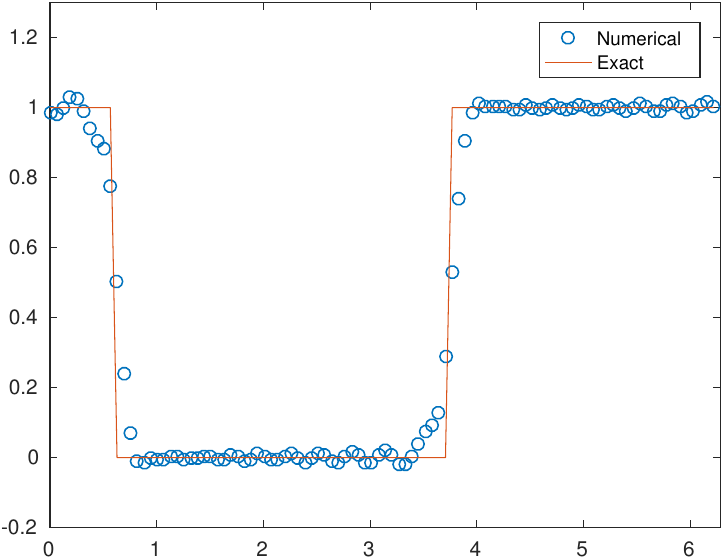}  }
 \hspace{.1in}
 \subfigure[with both limiters]{\includegraphics[scale=0.45]{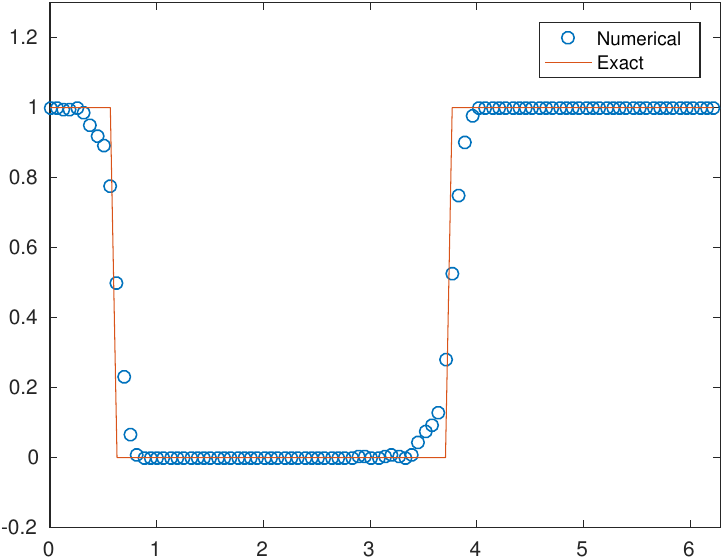}}
\caption{Linear convection  at $T=10$. Fourth order compact finite difference and fourth order SSP multistep with 
$\Delta t=\frac{1}{3}C_{ms}\Delta x$ and $100$ grid points. The TVB parameter in \eqref{minmod} is  $p=5$. }
\label{1dconvection_shockimg}
 \end{figure}

 \begin{figure}[ht]
 \subfigure[Without any limiter.]{\includegraphics[scale=0.45]{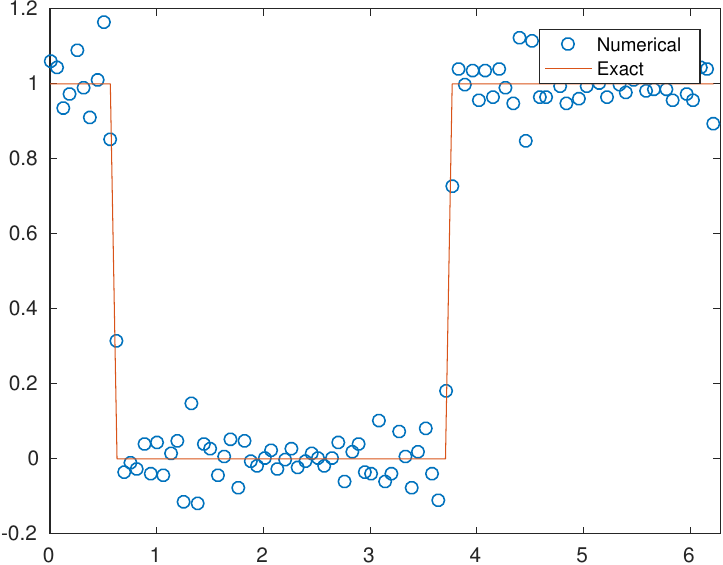} }
 \hspace{.1in}
 \subfigure[With the bound-preserving limiter.]{\includegraphics[scale=0.45]{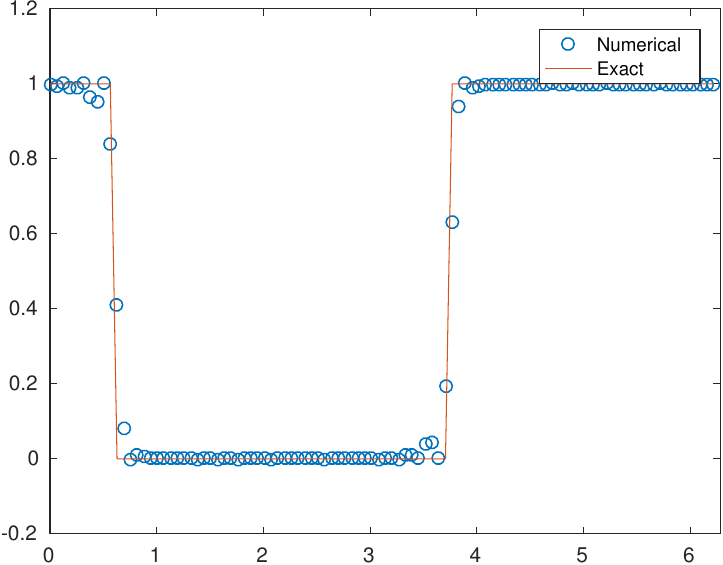}}
\caption{Linear convection at $T=10$. Eighth order compact finite difference and the fourth order SSP multistep method with 
$\Delta t=C_{ms} \frac{6}{25} \Delta x$ and $100$ grid points }
\label{1dconvection_shock8thimg}
 \end{figure}

\end{ex}
   
\begin{ex} {One dimensional Burgers' equation.} \label{1dburgerseqn}     

Consider the Burgers' equation $u_t+(\frac{u^2}{2})_x=0$ with a periodic boundary condition on $[-\pi,\pi]$.
For the initial data $u_0(x)=\sin(x)+0.5$, the exact solution is smooth up to $T=1$, then it develops a moving shock.
We list the errors of the fourth order scheme at $T=0.5$ in Table \ref{1dburgerstable}
where 
the time step is  $\Delta t=\frac{1}{3}C_{ms}\Delta x$ for SSP multistep and  $\Delta t=\frac{5}{3}C_{ms}\Delta x$ for SSP Runge-Kutta with $\Delta x=\frac{2\pi}{N}$. 
We observe the expected fourth order accuracy for the multistep time discretization. 
At $T=1.2$, the exact solution contains a shock near $x=-2.5$. The errors on the smooth region $[-2,\pi]$ at $T=1.2$ are listed in Table \ref{1dburgerstable2} where high order accuracy is lost. Some high order schemes can still be high order accurate on a smooth region away from the shock in this test, see \cite{zhang2010genuinely}. We emphasize that in all our numerical tests, Step III in Algorithm \ref{alg_limiter2} was never triggered. In other words, set of Class I is rarely encountered in practice. So the limiter Algorithm \ref{alg_limiter2} is a local three-point stencil limiter for this particular example rather than a global one. The loss of accuracy in smooth regions is possibly due to the fact that compact finite difference operator is defined globally thus the error near discontinuities will pollute the whole domain.

The solutions of the fourth order compact finite difference and the fourth order SSP multistep with the bound-preserving limiter and the TVB limiter at time $T=2$ are shown in Figure \ref{1dburgersimg}, for which the exact solution is in the range $[-0.5, 1.5]$. 
The  TVB limiter alone does not eliminate the overshoot or undershoot. When both the bound-preserving and the TVB limiters
are used, we can obtain a non-oscillatory bound-preserving numerical solution.
\begin{table}[ht]
\centering
\caption{The fourth order scheme with  limiter for the Burgers' equation. Smooth solutions. }
  \resizebox{\textwidth}{!}{   \begin{tabular}{|c|c c|c c|cc| c c|}
\cline{1-9}
& \multicolumn{4}{|c|} {Fourth order SSP multistep} & \multicolumn{4}{|c|} {Fourth SSP Runge-Kutta}  \\
\cline{1-9}
\hline  N &  $L^1$ error  &  order & $L^\infty$ error & order & $L^1$ error & order & $L^\infty$ error & order\\
\cline{1-9}
 $20$  & 6.92E-4 & - & 5.24E-3 & -& 7.79E-4 & - & 5.61E-3 & - \\
\cline{1-9}
 $40$  & 3.28E-5 & 4.40 & 3.62E-4 & 3.85 & 4.45E-5 & 4.13& 4.77E-4 & 3.56 \\
\cline{1-9}
 $80$  & 1.90E-6 &  4.11 & 2.00E-5 & 4.18 & 3.53E-6 & 3.66 & 2.09E-5 & 4.51  \\ 
 \cline{1-9}
 $160$  & 1.15E-6 & 4.04 & 1.24E-6 & 4.01 & 4.93E-7 & 2.84 & 5.47E-6 & 1.93  \\
 \cline{1-9}
 $320$  & 7.18E-9 & 4.00 & 7.67E-8 & 4.01 & 8.78E-8 & 2.49 & 1.73E-6 & 1.66 \\
\hline
\end{tabular}}
\label{1dburgerstable}
\end{table}

\begin{table}[ht]
\centering
\caption{Burgers' equation. The errors are measured in the smooth region away from the shock. }
  \resizebox{\textwidth}{!}{   \begin{tabular}{|c|c c|c c|cc| c c|}
\cline{1-9}
& \multicolumn{4}{|c|} {Fourth order SSP multistep} & \multicolumn{4}{|c|} {Fourth SSP Runge-Kutta}  \\
\cline{1-9}
\hline  N &  $L^1$ error  &  order & $L^\infty$ error & order & $L^1$ error & order & $L^\infty$ error & order\\
\cline{1-9}
 $20$  & 1.59E-2 & - & 5.26E-2 & - & 1.62E-2 & - & 5.39E-2 & - \\
\cline{1-9}
 $40$   & 2.10E-3 & 2.92& 1.38E-2 & 1.93 & 2.11E-3 & 2.94 & 1.39E-2 & 1.95\\
\cline{1-9}
 $80$ & 6.35E-4 & 1.73 & 6.56E-3 & 1.07   & 6.48E-4 &  1.70 & 7.01E-3 & 0.99\\ 
 \cline{1-9}
 $160$   & 1.48E-4 & 2.10 & 1.65E-3 & 1.99  & 1.51E-4 & 2.10 & 1.66E-3 & 2.08\\
 \cline{1-9}
 $320$   & 3.12E-5 & 2.25 & 6.10E-4 & 1.43 & 3.14E-5 & 2.26 & 6.13E-4 & 1.44\\
\hline
\end{tabular}}
\label{1dburgerstable2}
\end{table}

 \begin{figure}[ht]
 
 \subfigure[without any limiter]{\includegraphics[scale=0.45]{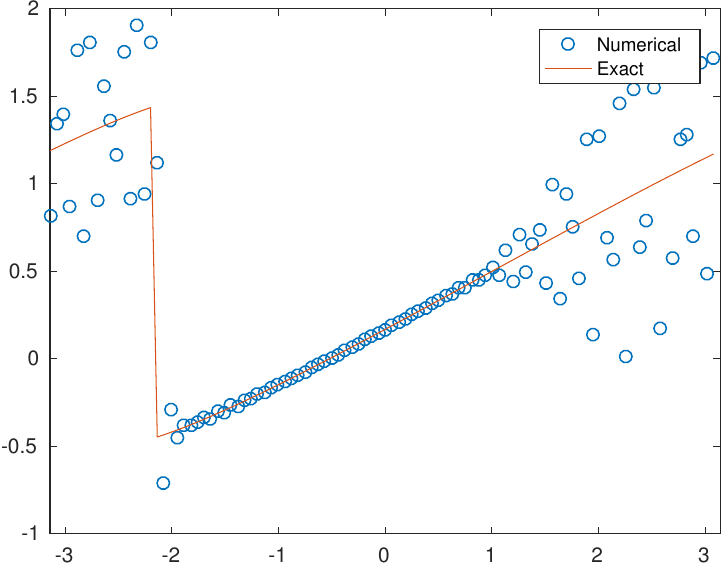} }
 \hspace{.1in}
 \hspace{.1in}
 \subfigure[with both limiters]{\includegraphics[scale=0.45]{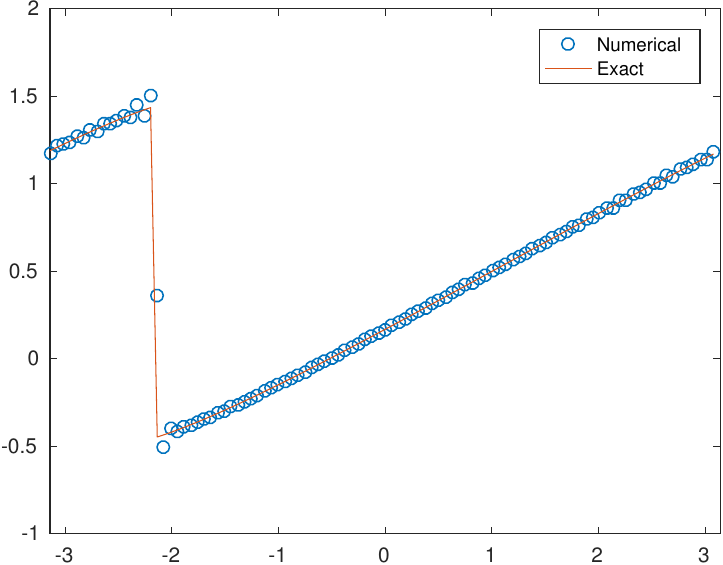}}
\caption{Burgers' equation at $T=2$. Fourth order compact finite difference with 
$\Delta t=\frac{1}{3\max_x|u_0(x)|}C_{ms}\Delta x$ and $100$ grid points. The TVB  parameter in \eqref{minmod} is set as $p=5$.
 }
 \label{1dburgersimg}
 \end{figure}

  \end{ex}
   
\begin{ex}{One dimensional convection diffusion equation.}\label{1dheateqn}

Consider the linear convection diffusion equation $u_t+cu_x=du_{xx}$ with a periodic boundary condition on $[0,2\pi]$.
For the initial $u_0(x)=\sin(x)$,
the exact solution is $u(x,t)=exp(-dt)sin(x-ct)$ which is in the range $[-1,1]$.
We set $c=1$ and $d=0.001$.
The errors of the fourth order scheme at $T=1$ are listed in the Table \ref{1dcondifftable} in which $\Delta t=C_{ms}min\{\frac{1}{6}\frac{\Delta x}{c},\frac{5}{24}\frac{\Delta x^2}{d}\}$ for SSP multistep 
and  $\Delta t=5C_{ms}min\{\frac{1}{6}\frac{\Delta x}{c},\frac{5}{24}\frac{\Delta x^2}{d}\}$ for SSP Runge-Kutta with $\Delta x=\frac{2\pi}{N}$. 
We observe the expected
fourth order accuracy for the SSP multistep method.
Even though the bound-preserving limiter is triggered, the order reduction for the Runge-Kutta method is not observed for the convection diffusion equation. 
One possible explanation is that the source of such an order reduction is due to the lower order accuracy of inner stages in the Runge-Kutta method, which 
is proportional to the time step. Compared to $\Delta t=\mathcal O(\Delta x)$ for a pure convection, the time step is $\Delta t=\mathcal O(\Delta x^2)$ in 
a convection diffusion problem thus the order reduction is much less prominent. 
See the Table \ref{1dcondifftable8th} for the errors at $T=1$ of the eighth order scheme  with $\Delta t=C_{ms}\min\{\frac{3}{25}\frac{\Delta x^2}{c},\frac{131}{530}\frac{\Delta x^2}{d}\}$ for SSP multistep
and $\Delta t=5C_{ms}\min\{\frac{3}{25}\frac{\Delta x^2}{c},\frac{131}{530}\frac{\Delta x^2}{d}\}$  for SSP Runge-Kutta where $\Delta x=\frac{2\pi}{N}$.

\begin{table}[ht]
\centering
\caption{The fourth order compact finite difference with  limiter for linear convection diffusion. }
\resizebox{\textwidth}{!}{    \begin{tabular}{|c|c c|c c|cc| c c|}
\cline{1-9}
& \multicolumn{4}{|c|} {Fourth order SSP multistep} & \multicolumn{4}{|c|} {Fourth order SSP Runge-Kutta}  \\
\cline{1-9}
\hline  N &  $L^1$ error  &  order & $L^\infty$ error & order & $L^1$ error & order & $L^\infty$ error & order\\
\cline{1-9}
 $20$  & 3.30E-5 & - & 5.19E-5 & -& 3.60E-5 & - & 6.09E-5 & - \\
\cline{1-9}
 $40$  & 2.11E-6 & 3.97 & 3.30E-6 & 3.97  & 2.44E-6 & 4.00 & 3.52E-6 & 4.12 \\
\cline{1-9}
 $80$  & 1.33E-7 &  3.99 & 2.09E-7 & 3.98  & 1.37E-7 & 4.04 & 2.15E-7 & 4.03  \\ 
 \cline{1-9}
 $160$  & 8.36E-9 & 3.99 & 1.31E-8 & 3.99 & 8.46E-9 & 4.02 & 1.33E-8 & 4.02  \\
 \cline{1-9}
 $320$  & 5.24E-10 & 4.00 & 8.23E-10 & 4.00 & 5.29E-10 & 4.00 & 8.31E-10 & 4.00 \\
\hline
\end{tabular}}
\label{1dcondifftable}
\end{table}

\begin{table}[ht]
\centering
\caption{The eighth order compact finite difference 
with  limiter for linear convection diffusion.}
\begin{tabular}{|c|c c|c c|cc| c c|}
\cline{1-9}
& \multicolumn{4}{|c|} {SSP multistep} & \multicolumn{4}{|c|} {SSP Runge-Kutta}  \\
\cline{1-9}
\hline  N &  $L^1$ error  &  order & $L^\infty$ error & order & $L^1$ error & order & $L^\infty$ error & order\\
\cline{1-9}
 $10$  & 3.85E-7 & - & 5.96E-7 & -& 3.85E-7 & - & 5.95E-7 & - \\
\cline{1-9}
 $20$  & 1.40E-9 & 8.10 & 2.20E-9 & 8.08  & 1.42E-9 & 8.08 & 2.23E-9 & 8.06\\
\cline{1-9}
 $40$  & 5.46E-12 &  8.01 & 8.60E-12 & 8.00  & 5.48E-12 & 8.02 & 8.69E-12 & 8.01  \\ 
 \cline{1-9}
 $80$  & 3.53E-12 & 0.63 & 6.46E-12 & 0.41 & 1.06E-12 & 2.37 & 3.29E-12 & 1.40  \\
\hline
\end{tabular}
\label{1dcondifftable8th}
\end{table}

\end{ex}

\begin{ex}{Nonlinear degenerate diffusion equations.}

 A representative test for validating the positivity-preserving property of a scheme solving nonlinear diffusion equations is the porous medium equation,
 $
  u_t=(u^m)_{xx}, m>1.
$
We consider the Barenblatt analytical solution given by
$$
 B_m(x,t)=t^{-k}[(1-\frac{k(m-1)}{2m}\frac{|x|^2}{t^{2k}})_+]^{1/(m-1)},
$$
where $u_+=\max\{u,0\}$ and $k=(m+1)^{-1}$. 
The initial data is the Barenblatt solution at $T=1$ with periodic boundary conditions on $\left[-6, 6\right]$. 
The solution is computed till time $T = 2$. 
High order schemes without any particular positivity treatment will generate negative solutions \cite{zhang2012maximum2, zhang2013maximum, Srinivasan2018positivity}.
See Figure \ref{pme} for solutions of the fourth order scheme and  
 the SSP multistep method with $\Delta t=\frac{1}{3m}C_{ms}\Delta x$ and $100$ grid points. Numerical solutions are 
strictly nonnegative. Without the bound-preserving limiter, negative values emerge near the sharp gradients. 

 \begin{figure}[ht]
 
 
 \subfigure[$m = 5.$]{\includegraphics[scale=0.45]{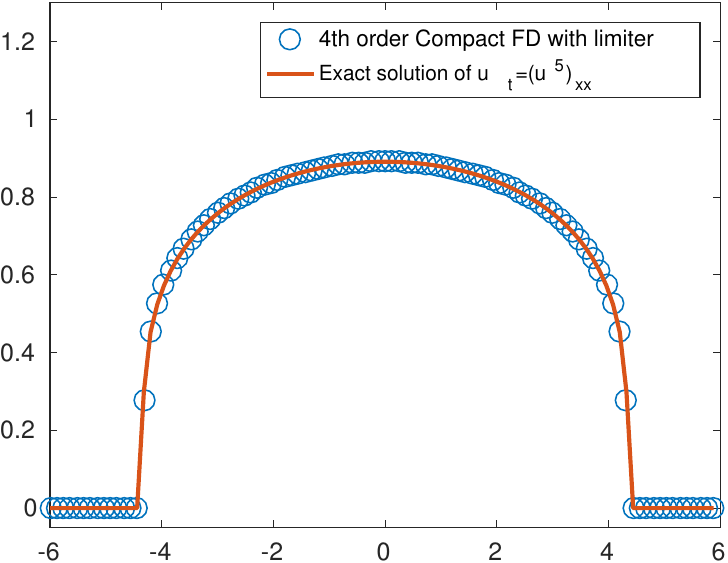}}
 \hspace{.15in}
 \subfigure[$m = 8.$]{\includegraphics[scale=0.45]{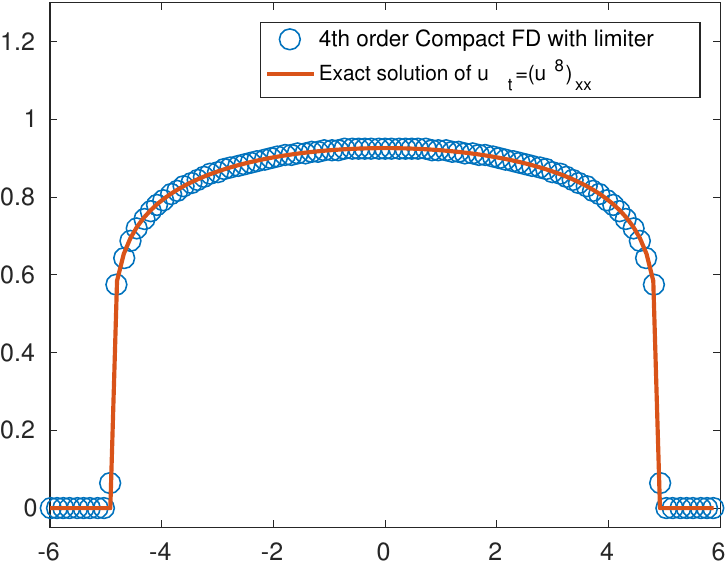}}
\caption{The fourth order compact finite difference  with  limiter for the  porous medium equation.}
\label{pme}
 \end{figure}
 
\end{ex}

\subsection{One-dimensional problems with non-periodic boundary conditions}

\begin{ex}{One-dimensional Burgers' equation with inflow-outflow boundary condition.}
Consider $u_t+(\frac{u^2}{2})_x=0$ on interval $[0,2\pi]$ with inflow-outflow boundary condition and smooth initial condition $u(x,0)=u_0(x)$.
Let $u_0(x)=\frac12\sin(x)+\frac12\geq 0$, we can set the left boundary condition as inflow $u(0,t)=L(t)$ and right boundary as outflow,
where $L(t)$ is obtained from the exact solution of initial-boundary value problem for the same initial data and a periodic boundary condition. We test the fourth order compact finite difference and fourth order SSP multistep method with the bound-preserving limiter.
The errors at  $T=0.5$ are listed in Table \ref{1dinflowoutflowtable} where $\Delta t=C_{ms}\Delta x$ and $\Delta x=\frac{2\pi}{N}$.
See  Figure \ref{1dinflowoutflowimg}  for the shock at  $T=3$ on a $120$-point grid with $\Delta t=C_{ms}\Delta x$.

\begin{table}[ht]
\centering
\caption{Burgers' equation. The fourth order scheme. Inflow and outflow boundary conditions.  }
  \begin{tabular}{|c|cc|cc|}
\hline  N &  $L^\infty$ error  &  order & $L^1$ error & order  \\
\hline
$20$  & 1.15E-4 & - & 7.80E-4 & - \\
\hline
  $40$  & 4.10E-6 &  4.81 & 2.00E-5 & 5.29\\ 
 \hline
  $80$  & 2.17E-7 & 4.24 & 9.43E-7 & 4.40\\
\hline
  $160$  & 1.22E-8 & 4.15 & 4.87E-8 & 4.28\\
\hline
  $320$ & 7.41E-10 & 4.05 & 2.87E-9 & 4.09\\
\hline
\end{tabular}
\label{1dinflowoutflowtable}
\end{table}

\begin{figure}[ht]
\centering
 \subfigure[Without any limiter.]{\includegraphics[scale=0.4]{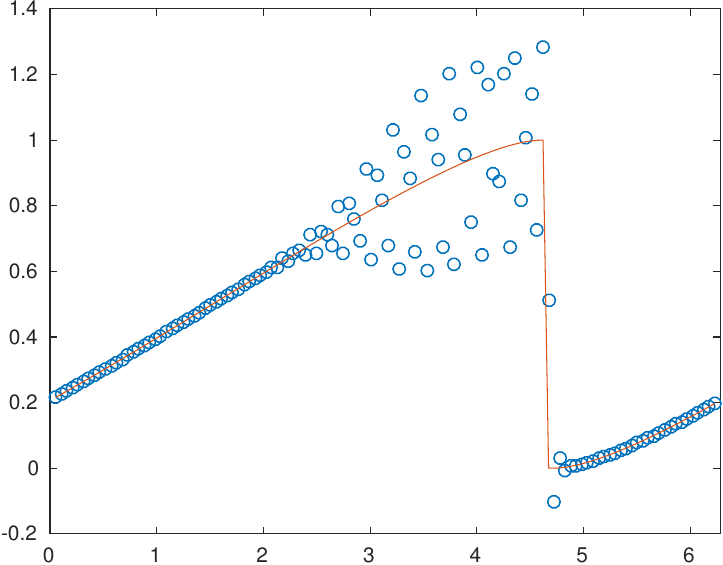}}
 \subfigure[With the bound-preserving limiter.]{\includegraphics[scale=0.4]{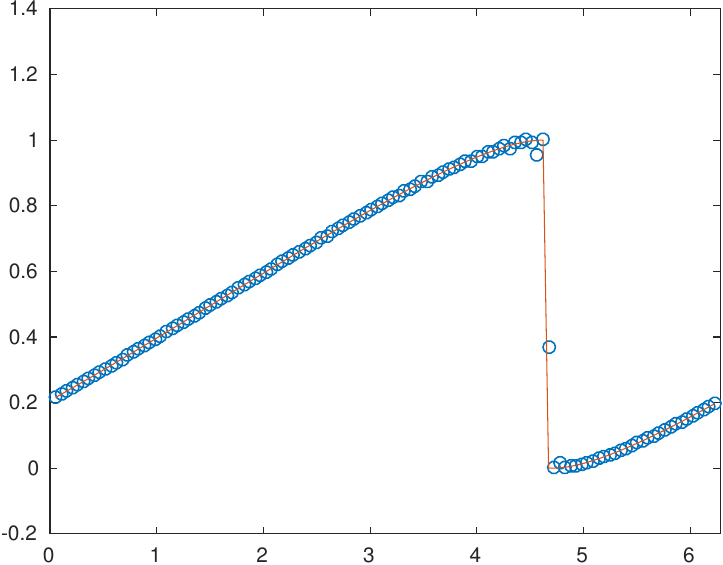} }
 \hspace{.15in}
 \caption{Burgers' equation. The fourth order scheme. Inflow and outflow boundary conditions. }
 \label{1dinflowoutflowimg}
\end{figure}
\end{ex}

\begin{ex}{One-dimensional convection diffusion equation with Dirichlet boundary conditions.}
 We consider equation $u_t+cu_x=du_{xx}$ on $[0,2\pi]$ with boundary conditions $u(0,t)=\cos(-ct)e^{-dt}$ and $u(2\pi,t)=\cos(2\pi-ct)e^{-dt}$.
The exact solution is $u(x,y,t)=\cos(x-ct)e^{-dt}$. We set $c=1$ and $d=0.01$.
We test the third order boundary scheme proposed in Section \ref{sec-Dirichlet} and the fourth order interior compact finite difference with the fourth order SSP multistep time discretization.
The errors at  $T=1$ are listed in Table \ref{1dbctable} where
 $\Delta t=C_{ms}\min\{\frac{4}{19}\frac{\Delta x}{c},\frac{695}{1596}\frac{\Delta x^2}{d}\}$, $\Delta x=\frac{2\pi}{N}$.

\begin{table}[ht]
\centering
\caption{A linear convection diffusion equation with Dirichlet boundary conditions. }
  \begin{tabular}{|c|cc|cc|}
\hline  N &  $L^\infty$ error  &  order & $L^1$ error & order  \\
\hline
$10$  & 1.68E-3 & - & 8.76E-3 & - \\
\hline
  $20$  & 1.47E-4 &  3.51 & 7.12E-4 & 3.62\\ 
\hline
  $40$  & 8.35E-6 &  4.14 & 4.27E-5 & 4.06\\ 
\hline
  $80$  & 4.44E-7 & 4.23 & 2.28E-6 & 4.23\\
\hline
  $160$  & 2.30E-8 & 4.27 & 1.10E-7 & 4.37\\
  \hline
\end{tabular}
\label{1dbctable}
\end{table}
\end{ex}

\subsection{Two-dimensional problems with periodic boundary conditions}
In this subsection we test the fourth order compact finite difference scheme solving two-dimensional problems with periodic boundary conditions. 

\begin{ex}{Two-dimensional linear convection equation.}\label{2dconvectioneqnexmp}
Consider $u_t+u_x+u_y=0$ on the domain $[0,2\pi]\times[0,2\pi]$ with a periodic boundary condition.
The scheme is tested with a smooth initial condition $u_0(x,y)=\frac12+\frac12\sin^4(x+y)$ to verify the accuracy. 
The errors at time $T=1$ are listed in Table \ref{table-2d-linear-accuracy} where $\Delta t=C_{ms} \frac16 \Delta x$ for the  SSP multistep method
and $\Delta t=5 C_{ms} \frac16 \Delta x$ for the SSP Runge-Kutta method with 
$\Delta x=\Delta y=\frac{2\pi}{N}$. 
We can observe the fourth order accuracy for the multistep method on resolved meshes and obvious order reductions for the Runge-Kutta method.

\begin{table}[ht]
\centering
\caption{Fourth order accurate compact finite difference  with  limiter 
for the 2D linear equation.}
\resizebox{\textwidth}{!}{  
\begin{tabular}{|c|c c|c c|cc| c c|}
\cline{1-9}
& \multicolumn{4}{|c|} {Fourth order SSP multistep} & \multicolumn{4}{|c|} {Fourth order SSP Runge-Kutta}  \\
\cline{1-9}
\hline  $N\times N$ Mesh&  $L^1$ error  &  order & $L^\infty$ error & order & $L^1$ error & order & $L^\infty$ error & order\\
\hline
\cline{1-9}
 $10\times 10$  & 4.70E-2 & - & 1.17E-1 & -& 8.45E-2 & - & 1.07E-1  & - \\
\cline{1-9}
 $20\times 20$  & 5.47E-3 & 3.10 & 8.97E-3 & 3.71 & 5.56E-3 & 3.93 & 9.09E-3 & 3.56\\
\cline{1-9}
$40 \times 40$  & 3.04E-4 &  4.17 & 5.09E-4 & 4.13 & 2.88E-4 & 4.27 & 6.13E-4 & 3.89  \\ 
\cline{1-9}
$80 \times 80$  & 1.78E-5 & 4.09 & 2.99E-5 & 4.09 & 1.95E-5 & 3.89 & 6.77E-5 & 3.18 \\
\cline{1-9}
$160 \times 160$ & 1.09E-6 & 4.03 & 1.85E-6 & 4.01 & 2.65E-6 & 2.88 & 1.26E-5 & 2.43\\
\hline
\end{tabular}}
\label{table-2d-linear-accuracy}
\end{table}

We also test the following discontinuous initial data:
$$u_0(x,y)=\left\{\begin{array}{ll}
 1, & \textrm{if $(x,y)\in[-0.2,0.2]\times[-0.2,0.2]$},\\
 0, & \textrm{otherwise}.
\end{array}\right.$$
The numerical solutions on a $80\times 80$ mesh at $T=0.5$ are shown in Figure \ref{fig-2d-shock} with 
$\Delta t=\frac{1}{6}C_{ms}\Delta x$ and $\Delta x=\Delta y=\frac{2\pi}{N}$.
Fourth order SSP multistep method is used.

\begin{figure}
\label{fig-2d-shock}
  \subfigure[Without any limiter.]{\includegraphics[scale=0.4]{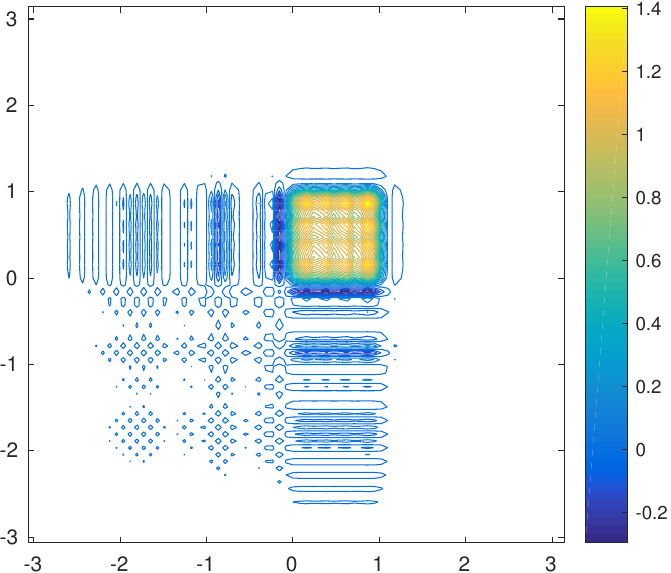}}
 \hspace{.1in}
   \subfigure[With bound-preserving limiter.]{\includegraphics[scale=0.4]{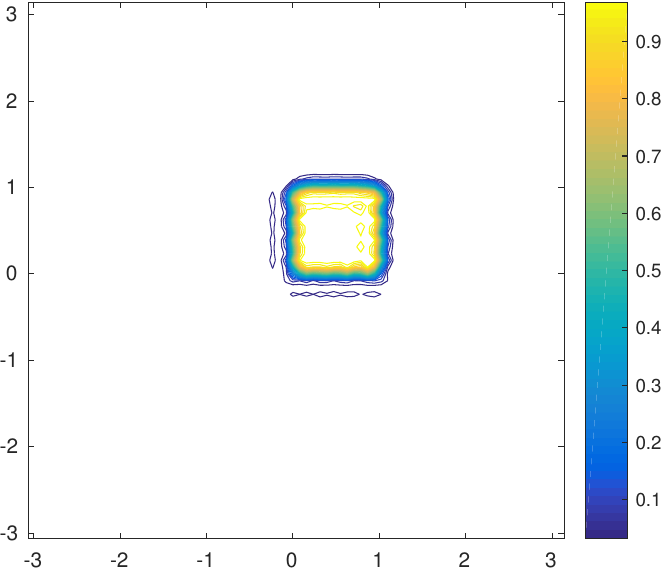}}
   \caption{Fourth order compact finite difference   for the 2D linear convection.}
\end{figure}

\end{ex}

\begin{ex}{Two-dimensional Burgers' equation.}
Consider $u_t+(\frac{u^2}{2})_x+(\frac{u^2}{2})_y=0$ with $u_0(x,y)=0.5+\sin(x+y)$
and periodic boundary conditions on  $[-\pi,\pi]\times[-\pi,\pi]$. 
At time $T=0.2$, the solution is smooth and the errors at $T=0.2$ on a $N\times N$ mesh are shown in the Table \ref{2dburgerstable} in which $\Delta t=C_{ms}\frac{\Delta x}{6\max_x |u_0(x)|}$ for multistep 
and $\Delta t=5 C_{ms}\frac{\Delta x}{6\max_x |u_0(x)|}$ for Runge-Kutta 
with $\Delta x=\Delta y=\frac{2\pi}{N}$.
At time $T=1$, the exact solution contains a shock. The numerical solutions of the fourth order SSP multistep method on a $100\times 100$ mesh are shown in Figure \ref{2dburgersimg} where $\Delta t=\frac{1}{6\max_x|u_0(x)|}C_{ms}\Delta x$.
The bound-preserving limiter ensures the solution to be in the range $[-0.5,1.5]$.

\begin{table}[ht]
\centering
\caption{Fourth order compact finite difference scheme with the bound-preserving limiter for the 2D Burgers' equation.}
\resizebox{\textwidth}{!}{  
\begin{tabular}{|c|c c|c c|cc| c c|}
\cline{1-9}
& \multicolumn{4}{|c|} {SSP multistep} & \multicolumn{4}{|c|} {SSP Runge-Kutta}  \\
\cline{1-9}
\hline  $N\times N$ Mesh &  $L^1$ error  &  order & $L^\infty$ error & order & $L^1$ error & order & $L^\infty$ error & order\\
\cline{1-9}
 $10\times 10$  & 1.08E-2 & - & 4.48E-3 & -& 9.16E-3 & - & 3.73E-2 & - \\
\cline{1-9}
 $20\times 20$  & 4.73E-4 & 4.52 & 3.76E-3 & 3.58 & 2.90E-4 & 4.98 & 2.14E-3 & 4.12 \\
\cline{1-9}
 $40 \times 40$  & 1.90E-5 &  4.64 & 1.45E-4 & 4.69 & 2.03E-5 & 3.83 & 1.12E-4 & 4.25 \\ 
 \cline{1-9}
$80 \times 80$  & 9.99E-7 & 4.25 & 7.43E-6 & 4.29 & 2.35E-6 & 3.12 & 1.54E-5 & 2.86  \\
 \cline{1-9}
$160 \times 160$  & 5.87E-8 & 4.09 & 4.26E-7 & 4.13 & 3.62E-7 & 2.70 & 5.13E-6 & 1.59 \\
\hline
\end{tabular}}
\label{2dburgerstable}
\end{table}

 \begin{figure}[ht]
 \subfigure[Without the bound-preserving limiter]{\includegraphics[scale=0.3]{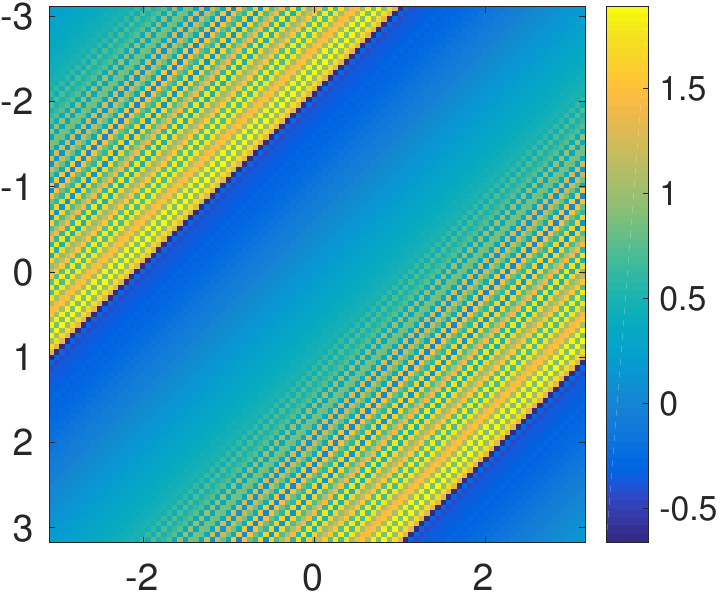} }
 \hspace{.1in}
 \subfigure[With the bound-preserving limiter]{\includegraphics[scale=0.3]{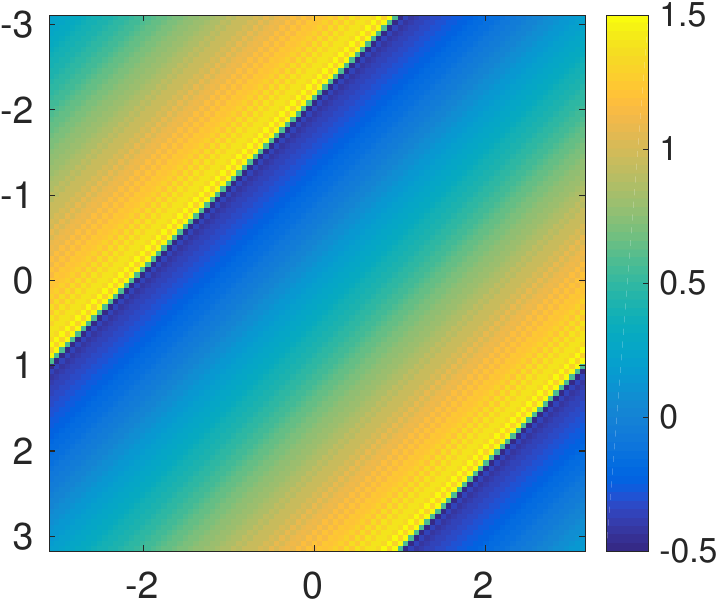}}
 \hspace{.1in}
  \subfigure[The exact solution]{\includegraphics[scale=0.3]{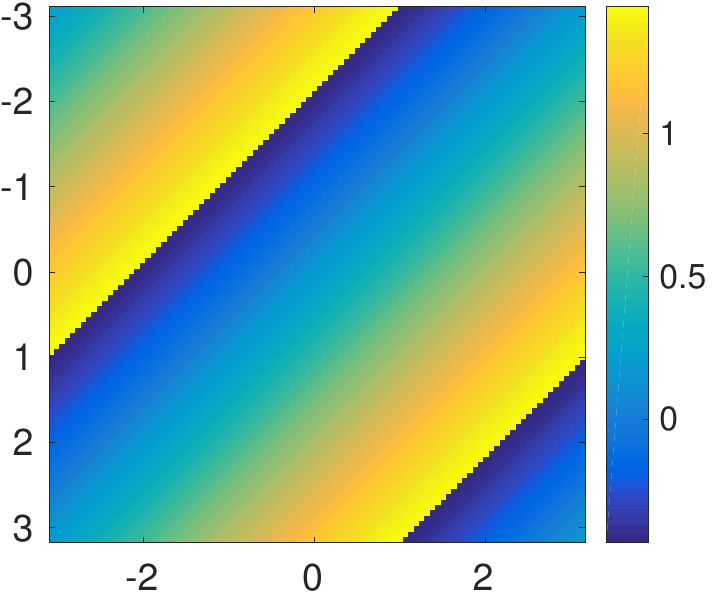}}
\caption{ The fourth order scheme. 2D Burgers' equation.}
\label{2dburgersimg}
 \end{figure}
\end{ex}

\begin{ex}{Two-dimensional convection diffusion equation.}

 Consider the equation $u_t+c(u_x+u_y)=d(u_{xx}+u_{yy})$ with $u_0(x,y)=\sin(x+y)$ and a 
 periodic boundary condition on  $[0,2\pi]\times[0,2\pi]$. 
The errors at time $T=0.5$ for $c=1$ and $d=0.001$ are listed in Table \ref{table-2d-conffu-accuracy}, in which
 $\Delta t=C_{ms}\min\{\frac{\Delta x}{6c},\frac{5\Delta x^2}{48d}\}$ for the  fourth-order SSP multistep method,
and $\Delta t=5 C_{ms}\min\{\frac{\Delta x}{6c},\frac{5\Delta x^2}{48d}\}$ for the fourth-order SSP Runge-Kutta method, where $\Delta x=\Delta y=\frac{2\pi}{N}$. 
\begin{table}[ht]
 \centering
\caption{Fourth order compact finite difference with  limiter for the 2D convection diffusion equation.}
\resizebox{\textwidth}{!}{  
\begin{tabular}{|c|c c|c c|cc| c c|}
\cline{1-9}
& \multicolumn{4}{|c|} {Fourth order SSP multistep} & \multicolumn{4}{|c|} {Fourth order  SSP Runge-Kutta}  \\
\cline{1-9}
\hline  N &  $L^1$ error  &  order & $L^\infty$ error & order & $L^1$ error & order & $L^\infty$ error & order\\
\cline{1-9}
 $10\times 10$  & 6.26E-4 & - & 9.67E-4 & -& 6.68E-4 & - & 9.59E-4 & - \\
\cline{1-9}
 $20\times 20$  & 3.62E-5 & 4.11 & 5.61E-5 & 4.11 & 3.60E-5 & 4.21 & 6.09E-5 & 3.98 \\
\cline{1-9}
 $40 \times 40$  & 2.20E-6 &  4.04 & 3.45E-6 & 4.02 & 2.24E-6 & 4.00 & 3.52E-6 & 4.12 \\ 
 \cline{1-9}
$80 \times 80$  & 1.35E-7 & 4.02 & 2.13E-7 & 4.01 & 1.37E-7 & 4.04 & 2.15E-7 & 4.03  \\
 \cline{1-9}
$160 \times 160$  & 8.45E-9 & 4.01 & 1.33E-8 & 4.01 & 8.46E-9 & 4.02 & 1.33E-8 & 4.02 \\
\hline
\end{tabular}
}
\label{table-2d-conffu-accuracy}
\end{table}
\end{ex}

\begin{ex}{Two-dimensional porous medium equation.}

We consider the equation $u_t=\Delta (u^m)$ with the following initial data 
\[u_0(x,y)=\left\{\begin{array}{ll}
 1, & \textrm{if $(x,y)\in[-0.5,0.5]\times[-0.5,0.5]$,}\\
 0, & \textrm{if $(x,y)\in[-2,2]\times[-2,2]/[-1,1]\times[-1,1]$,}
\end{array}\right.\]
and a periodic boundary condition on domain $[-2,2]\times[-2,2]$.
See Figure \ref{2dpmecontour} for the solutions at time $T=0.01$ for
SSP multistep method with $\Delta t=\frac{5}{48\max_x|u_0(x)|}C_{ms}\Delta x$ and $\Delta x=\Delta y=\frac{1}{15}$. The numerical solutions are strictly non-negative,
which is nontrivial for high order accurate schemes.
High order schemes without any positivity treatment will generate negative solutions in this test, see \cite{zhang2012maximum2, zhang2013maximum, Srinivasan2018positivity}.

 \begin{figure}[ht]
 \subfigure[$m=3$.]{\includegraphics[scale=0.33]{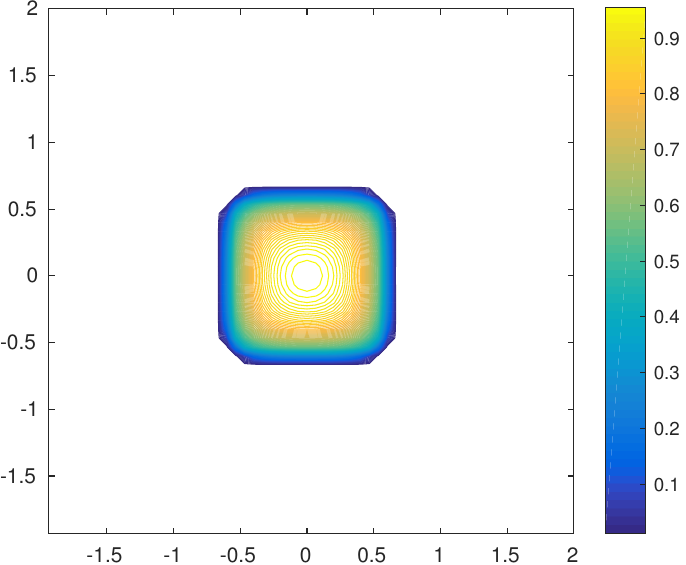} }
 \hspace{.1in}
 \subfigure[$m=4$.]{\includegraphics[scale=0.33]{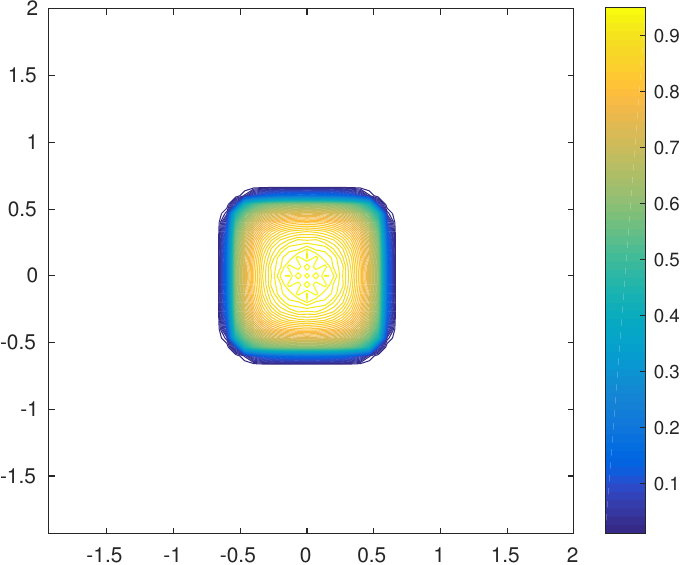}}
 \hspace{.1in}
  \subfigure[$m=5$.]{\includegraphics[scale=0.33]{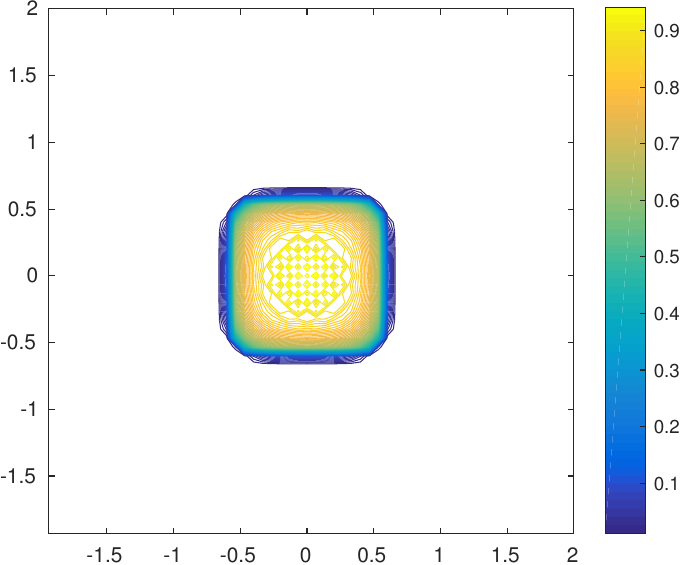}}
\caption{The fourth order scheme with limiter for 2D porous medium equations  $u_t=\Delta (u^m)$. }
\label{2dpmecontour}
 \end{figure}
\end{ex}

\section{Concluding remarks}
\label{sec-remark}

In this paper we have demonstrated that fourth order accurate compact finite difference schemes for convection diffusion problems with
periodic boundary conditions satisfy a weak monotonicity property, and a simple
three-point stencil limiter can enforce bounds without destroying the global conservation. Since the limiter is designed based on
an intrinsic property in the high order finite difference schemes, the accuracy of the limiter can be easily justified. 
This is the first time that the weak monotonicity is established for a high order accurate finite difference scheme, complementary to results regarding
the weak monotonicity property of high order finite volume and discontinuous Galerkin schemes in \cite{zhang2010maximum,zhang2011maximum, zhang2012maximum}.

We have discussed extensions to two dimensions, higher order accurate schemes and general boundary conditions, for which the five-diagonal weighting matrices can be factored 
as a product of tridiagonal matrices so that the same simple three-point stencil bound-preserving limiter can still be used. 
We have also proved that the TVB limiter in \cite{cockburn1994nonlinearly} does not affect the bound-preserving property. Thus with both the TVB and 
the bound-preserving limiters, the numerical solutions of high order compact finite difference scheme can be rendered non-oscillatory and  strictly bound-preserving without
losing accuracy and global conservation.
Numerical results suggest the good performance of the high order bound-preserving compact finite difference schemes.  

For more generalizations and applications,
there are certain complications. For using compact finite difference schemes on non-uniform meshes, one popular approach is to introduce a mapping to a uniform grid but such a mapping results in an extra variable coefficient which 
may affect the weak monotonicity. Thus any extension to non-uniform grids is much less straightforward. 
For applications to systems, e.g., preserving positivity of density and pressure in compressible Euler equations, the weak monotonicity can be easily extended to a weak positivity property. However, the same three-point stencil limiter cannot enforce the positivity for pressure. One has to construct a new limiter for systems.

\bibliographystyle{amsplain}
\bibliography{positivity.bib}

\end{document}